\documentclass[11pt]{article}
\usepackage[T1]{fontenc}
\usepackage[utf8]{inputenc}
\usepackage{authblk}
 \usepackage{graphics,wrapfig}
 \usepackage{flafter}
 \usepackage{amsmath,amsthm,amsfonts,amssymb,epsfig}
 \usepackage{cases}
 \numberwithin{equation}{section}
 \usepackage{leftidx}
 \usepackage{mathrsfs}
 \usepackage{bbding}
\usepackage{fancyhdr}
\usepackage{mathrsfs}
\usepackage[toc,page,title,titletoc,header]{appendix}
\usepackage{color}
\usepackage{subfigure}
\usepackage{stmaryrd}
\usepackage{latexsym}
\usepackage{psfrag}
\usepackage{graphicx,subfigure}
\usepackage{fancyhdr,graphicx}
\usepackage{multicol}
\usepackage{dsfont}
\usepackage{bbm}
\usepackage{booktabs}
\usepackage[center]{caption2}
\usepackage{cite}
\usepackage{multirow,makecell}
\usepackage{indentfirst}
\usepackage{appendix}
\usepackage{lineno}
\allowdisplaybreaks

\date{}
\newtheorem{theorem}{\bf{Theorem}}[section]

\newtheorem{lemma}{\bf {Lemma}}[section]

\newtheorem{remark}{\bf{Remark}}[section]
\newtheorem{example}{\bf{Example}}[section]
\newcommand{\be}{\begin{equation}}
\newcommand{\ee}{\end{equation}}
\newcommand\bea{\begin{eqnarray}}
\newcommand\eea{\end{eqnarray}}
\newcommand{\bean}{\begin{eqnarray*}}
\newcommand{\eean}{\end{eqnarray*}}
\newcommand{\ds}{\displaystyle}

\begin{document}
\title{$H^1$-analysis of H3N3-2\textbf{$_\sigma$}-based difference method for fractional hyperbolic equations\footnote{RD was partially supported by the National Natural Science Foundation of China (No. 12201076) and the China Postdoctoral Science Foundation (No. 2023M732180); CL was partially supported by the National Natural Science Foundation of China (No. 12271339).}}

\author[,1]{Rui-lian Du\thanks{E-mail address: drl158@shu.edu.cn}}
\affil[1]{\,Department of Mathematics, Shanghai University, and Newtouch Center for Mathematics of Shanghai University, Shanghai 200444, China}

\author[,1]{Changpin Li\thanks{Corresponding author; E-mail address: lcp@shu.edu.cn}}


\author[,2]{Zhi-zhong Sun\thanks{E-mail address: zzsun@seu.edu.cn}}

\affil[2]{\,School of Mathematics, Southeast University, Nanjing 211189, China}

\maketitle
\begin{abstract}
A novel H3N3-2$_\sigma$ interpolation approximation for the Caputo fractional derivative of order $\alpha\in(1,2)$ is derived in this paper, which improves the popular L2C formula with (3-$\alpha$)-order accuracy. By an interpolation technique, the second-order accuracy of the truncation error is skillfully estimated. Based on this formula, a finite difference scheme with second-order accuracy both in time and in space is constructed for the initial-boundary value problem of the time fractional hyperbolic equation. It is well known that the coefficients' properties of discrete fractional derivatives are fundamental to the numerical stability of time fractional differential models. We prove the related properties of the coefficients of the H3N3-2$_\sigma$ approximate formula. With these properties, the numerical stability and convergence of the difference scheme are derived immediately by the energy method in the sense of $H^1$-norm. Considering the weak regularity of the solution to the problem at the starting time, a finite difference scheme on the graded meshes based on H3N3-2$_\sigma$ formula is also presented. The numerical simulations are performed to show the effectiveness of the derived finite difference schemes, in which the fast algorithms are employed to speed up the numerical computation.

\vskip5mm

\end{abstract}

\section{Introduction}
\label{intro}
In the past few decades, it has been shown that differential equations with fractional derivatives can be applied to accurate modeling many physical processes. Therefore, much attention has been paid to the numerical analyses of such problems, and many interesting works have emerged \cite{dengwh,mengxy,wangdl,zhoutao}. A series of mature results have been obtained for fractional differential equations with Caputo time fractional derivatives, see \cite{alikhanov,Shenjie,sunzz2,chijiang,du2020,xucj,licai2019,Mclean,yan2023,sunzz1,Wanglilian,licp,Ainsworth,Dimitrov}
and lots of references cited therein. Typically, fractional hyperbolic (wave) equations with Caputo derivative have become widely used models for describing power law attenuation behavior of waves in complex inhomogeneous media.

In the current paper, we focus on the following initial-boundary value problem of the  fractional hyperbolic (or fractional diffusion-wave) equation
\cite{luchko}
\begin{align}
& _C\mathrm{D}_{0,t}^\alpha u(x,t)=u_{xx}(x,t) +f(x,t),\quad & (x,t)\in (0,L)\times(0,T],\label{eq}
\\[0.01in]
& u(x,0)=\varphi(x),\quad u_t(x,0)=\psi(x),\quad  &x\in(0,L),\label{cz}
\\[0.01in]
& u(0,t)=0,\quad u(L,t)=0,\quad & t\in (0,T],\label{bz}
\end{align}
in which $f,\varphi,\psi$ are given suitably smooth functions with consistency conditions $\varphi(x)=\psi(x)=0$ at the end points $x=0, L.$ The Caputo derivative $_C\mathrm{D}_{0,t}^\gamma\, p(t)$ is defined by
\begin{equation}\label{caputo}
_C\mathrm{D}_{0,t}^\gamma\, p(t)=\frac{1}{\Gamma(n-\gamma)}\int_0^t\frac{p^{(n)}(s)}
{(t-s)^{\gamma-n+1}}\mathrm{d}s,\quad \gamma\in(n-1,n),\quad 0<t<T.
\end{equation}
In \eqref{eq}, we set $\alpha\in(1,2)$ (i.e., when $n=2$ in \eqref{caputo}), then $p(t)\in AC^2[0,T]$ is often presumed, which is just a sufficient condition such that $_C\mathrm{D}_{0,t}^\alpha\, p(t)$ exists for $\alpha\in(1,2).$ Eqs. \eqref{eq}-\eqref{bz} can be applied to depicting evolution processes intermediate between standard diffusion and wave propagation \cite{luchko}. For example, it governs the transmission
of mechanical waves in viscoelastic media \cite{Wanghong}, the image processing \cite{ZhangW}, the mechanical response \cite{Nigmatullin} and so on.

Note that \eqref{eq} with $\alpha=2$ reduces to the standard second-order hyperbolic PDE \cite{pap-e} that models the undamped motion of the perfectly elastic material. In this case, the $u_{tt}$ term presents the inertial force of elastic material of the unit density, the Laplacian term accounts for the impact of the internal force, while $f$ represents the external loading. However, many experiments reported in the literature showed that the materials do not behave purely elastically but also demonstrate certain internal dissipation mechanisms, and so exhibit both stored and dissipative energy components with nonlocal memory effect. Hence, the conventional model given by the $u_{tt}$ term does not properly describe the damping effect in the current context. Instead, a fractional time derivative term was adopted to improve the mathematical model \cite{spanos-p}.

In general, one of the key points of constructing numerical methods for solving fractional differential equations is to discretetize fractional operators. Up to now, there have been quite a number of numerical methods for approximating the Caputo derivative of order in (0,1) (i.e., when $n=1$ in \eqref{caputo}) based on the idea of piecewise interpolations. Specifically, when $n=1,$ we rewrite \eqref{caputo} at a generic grid point $t=t_k$ as
\[_C\mathrm{D}_{0,t}^\gamma\, p(t)|_{t=t_k}=\frac{1}{\Gamma(1-\gamma)}\sum_{j=1}^k\int_{t_{j-1}}^{t_j}\frac{p'(s)}
{(t-s)^{\gamma}}\mathrm{d}s.\]
As can be seen from the above expression, the original integral interval is divided into several cells, and then we can obtain different numerical methods by adopting different interpolate formulae on each cell. A widely used method with (2-$\gamma$)-th order accuracy, called L1 formula \cite{jinbt,Kopteva,liao-li-zhang,sunzz1,xucj}, was presented by using the linear interpolation on overall intervals. In \cite{sunzz2}, the authors proposed the L1-2 formula with the order 3-$\gamma$ by using the quadratics interpolation on all intervals except the linear interpolation on the first interval. Alikhanov \cite{alikhanov} established the L2-1$_\sigma$ formula with (3-$\gamma$)-th order accuracy by substituting the final point $t=t_k$ for $t=t_{k+\sigma} (\sigma=1-\gamma/2\in(1/2,1)),$ and applied the quadratics interpolation on all the subintervals while the linear interpolation on the last interval $[t_k, t_{k+\sigma}].$ Later, the authors in \cite{xucj1} built the L2 formula with (3-$\gamma$)-th order accuracy by adopting the quadratic interpolation polynomial on all the subintervals. In practical applications, the ideal result is to obtain a high-accuracy numerical scheme with complete theoretical analyses. As can be seen from the above results, the accuracies of the L1-2 formula, L2-1$_\sigma$ formula and L2 formula are higher than that of L1 formula, nevertheless, the stability and convergence of L1-2-based difference schemes were not proved yet, and that of L2-based difference schemes are more complicated and not easy to generalize, L2-1$_\sigma$-based numerical methods is pithy and efficient, so it has been widely used.

To the best of our knowledge, there are still challenges for numerical methods of the time fractional hyperbolic equations (when $n=2$ in \eqref{caputo}) compared to that for the sub-diffusion equations (when $n=1$ in \eqref{caputo}). At present, some commonly used techniques are the method of order reduction. That is, by introducing an intermediate variable $v=u_t,$ the original governing equation \eqref{eq} is transformed into a lower order coupled system:
\begin{align}
& _C\mathrm{D}_{0,t}^{\alpha-1} v(x,t)=u_{xx}(x,t) +f(x,t),\quad & (x,t)\in (0,L)\times(0,T],\label{eqv}
\\[0.01in]
& v(x,t)=u_t(x,t),\quad  &(x,t)\in (0,L)\times(0,T].\notag
\end{align}
We note that the time-fractional derivative on the auxiliary function
$v$ in Eq. \eqref{eqv} is of order $\alpha-1$ which belongs to $(0, 1).$ Thus the time fractional hyperbolic equation
can be solved following the standard framework of the L1 method on the uniform temporal
meshes \cite{sunzz1}, L1 method on the graded meshes \cite{sun-zhao,shen-martin-sun} and L2-1$_\sigma$ method on the uniform meshes \cite{sunh,sunh1}. On the flip side, with the help of the properties of fractional derivatives, authors in \cite{lyu-p} proposed a new order reduction method by introducing a novel auxiliary function $\mathbf{v}=\,_C\mathrm{D}_{0,t}^{\frac{\alpha}{2}} \mathbf{u},$ where $\mathbf{u}=u-t\psi,$ the original equation \eqref{eq} is then converted to the following coupled system:
\begin{align*}
& _C\mathrm{D}_{0,t}^{\frac{\alpha}{2}} \mathbf{v}(x,t)=u_{xx}(x,t) +f(x,t)+t\psi,\quad & (x,t)\in (0,L)\times(0,T],
\\[0.01in]
& \mathbf{v}(x,t)=\,_C\mathrm{D}_{0,t}^{\frac{\alpha}{2}} \mathbf{u}(x,t),\quad  &(x,t)\in (0,L)\times(0,T],
\end{align*}
where $u=\mathbf{u}+t\psi.$ Based on the widespread of L1 formula and L2-1$_\sigma$ formula, the authors designed the temporal optimal convergence orders with (2-$\alpha/2)$-th order for the L1 algorithm and 2nd order for the Alikhanov algorithm respectively. However, there are two shortcomings for these methods: (1) Coupling system is not of
structure consistency in the viewpoint of time derivative order; (2) The increase in the number of equations makes the theoretical analysis of their numerical schemes more complicated and so is not easy to generalize.

For the above reasons, it is desirable to develop high-order numerical schemes without changing the structure of the time fractional hyperbolic equation \eqref{eq}. That is, discretize $((1,2)\ni)\alpha$-th order Caputo derivatives directly. Similarly, we rewrite the Caputo derivative at the point $t=t_k$ in Eq. \eqref{eq} as a sum of integrals, that is
\begin{equation}\label{caputo2}
_C\mathrm{D}_{0,t}^\alpha\, p(t)|_{t=t_k}=\frac{1}{\Gamma(2-\alpha)}\sum_{j=1}^k\int_{t_{j-1}}^{t_j}\frac{p''(s)}
{(t-s)^{\alpha-1}}\mathrm{d}s,\quad 1<\alpha<2.
\end{equation}
Up to now, however, there is very few discussion on this subject. Lynch et al. \cite{lynch} approximated the $p''(s)$ on each interval by the second central difference quotient to get L2 formula, and used a four-point discretization of the second derivative in the integrand to obtain a variant form of L2 formula, called L2C formula, but without theoretical analysis yet. Recently, Shen et al. \cite{shenjy} proposed an H2N2 interpolation for Caputo derivative defined in \eqref{caputo2} by adopting the Hermite quadratic interpolation polynomial on the first interval and quadratic Newton interpolation polynomials on the rest of intervals, the stability and convergence of the numerical scheme are analysed further. It is remarkable that the numerical accuracies of the mentioned numerical schemes are lower than 2. This motivates us to develop a high-order numerical formula for $((1,2)\ni)\alpha$-th order Caputo derivative. For any fixed $T,$ take a positive integer $N,$ and denote $\tau=T/N,$ $t_k=k\tau,~0\leq k\leq N.$ Inspired by the Alikhanov formula, for the first time, we consider $\alpha\,(1<\alpha<2)$-order Caputo derivative at a new super-convergence point $t=t_{k+\sigma},$ here $\sigma=1-\alpha/2 \in(0,1/2),$ we rewrite the integral in Caputo derivative as a sum of integrals, that is
\begin{align}\label{caputo-m2}
_C\mathrm{D}_{0,t}^\alpha\, p(t)|_{t=t_{k+\sigma}}
=&\ds\frac{1}{\Gamma(2-\alpha)}\Big[\int_{t_0}^{t_\frac{1}{2}}
\frac{p''(s)}{(t_{k+\sigma}-s)^{\alpha-1}}\mathrm{d}s+
\sum_{l=1}^{k-1}\int_{t_{l-\frac{1}{2}}}^{t_{l+\frac{1}{2}}}
\frac{p''(s)}{(t_{k+\sigma}-s)^{\alpha-1}}\mathrm{d}s\notag\\
&+\int_{t_{k-\frac{1}{2}}}^{t_{k+\sigma}}\frac{p''(s)}{(t_{k+\sigma}-s)
^{\alpha-1}}\mathrm{d}s\Big].
\end{align}
We adopt cubic Hermite
interpolation polynomial (H3) on first interval $[t_0, t_{\frac{1}{2}}],$ cubic Newton interpolation
polynomial (N3) on the interval $[t_{l-\frac{1}{2}},t_{l+\frac{1}{2}}]\,(1\leq l\leq k-1),$ and square Newton
interpolation polynomial (N$2_\sigma)$ on final interval $[t_{k-\frac{1}{2}}, t_{k+\sigma}]$ to obtain a novel formula, called H3N3-2$_\sigma$ formula. 

This paper aims to derive a novel H3N3-2$_\sigma$ approximation formula for Caputo derivative of order $\alpha\in(1,2)$ with second-order accuracy. Then, the coefficients' properties of the novel formula are presented. We further prove the unconditional numerical stability and convergence of the finite difference scheme in the sense of $H^1$-norm via the discrete energy technique. Through these studies, we hope that the task of numerical stability analysis for fractional hyperbolic equations can be put on the agenda. In addition, the solutions of time fractional differential equations likely exhibit initial weak regularity \cite{jinbt,sakamoto,Martin}, the methods on nonuniform time meshes for time fractional differential equations have gained much attention, especially on the graded meshes \cite{Kopteva,liao-li-zhang}. This motivates us to consider an H3N3-2$_\sigma$-based difference scheme on the graded meshes for problem \eqref{eq}-\eqref{bz}.

The structure of the rest paper is outlined as follows. In Section \ref{appro}, an H3N3-2$_\sigma$ numerical approximate formula for the Caputo derivative defined in \eqref{caputo-m2} is derived and a sharp error estimate for the formula is obtained. An H3N3-2$_\sigma$-based difference scheme for problem \eqref{eq}-\eqref{bz} is constructed in Section \ref{scheme}, and the numerical stability and convergence of the difference scheme in the sense of $H^1$-norm are also strictly proved. In Section \ref{sec-regularity}, a difference scheme on the graded meshes is showed for problem \eqref{eq}-\eqref{bz} when the solution has initial weak regularity. The numerical examples are provided to confirm the theoretical findings and the effectiveness of the finite difference schemes in Section \ref{example}. Finally, the brief conclusions and remarks are given in the last section.

\section{H3N3-2$_\sigma$ approximation}\label{appro}
In the section, we will derive an H3N3-2$_\sigma$ approximation to Caputo derivative of order $\alpha\in(1,2)$ defined in \eqref{caputo-m2}.

Let $\Omega_\tau=\{t_l\,|\,l=0,1,2,\cdots,N\},$ where $t_0=0,\,t_N=T.$ For a given function $p\in C^1[0,T],$ denote $p^k=p(t_k).$ And introduce the following notations
\[\delta_tp^{k-\frac{1}{2}}=\frac{p^k-p^{k-1}}{\tau},\quad \delta_t^2p^0=\frac{2}{\tau}[\delta_tp^{\frac{1}{2}}-p'(t_0)],\quad \delta_t^2p^k=\frac{\delta_tp^{k+\frac{1}{2}}-\delta_tp^{k-\frac{1}{2}}}{\tau}, \quad k\geq1. \]

Next, suppose $p\in C^{4}[0,T].$ Consider the approximation of function $p$ on the first interval $[0,t_{\frac{1}{2}}].$ Using four data $(t_0,p^0),$ $ (t_0,p'(t_0)), (t_1,p^1)$ and $(t_2,p^2),$ we can obtain a cubic Hermite interpolation polynomial ($\mathrm{H3}$)
\[H_{3,0}(t)=p(t_0)+p'(t_0)(t-t_0)+\frac{1}{2}\delta_t^2p^0(t-t_0)^2
+\frac{1}{4\tau}\big(\delta_t^2p^1-\delta_t^2p^0\big)(t-t_0)^2(t-t_1).\]
It immediately follows that
\begin{equation}\label{H3}
H''_{3,0}(t)=\delta_t^2p^1\frac{t-t_{\frac{1}{3}}}{t_1-t_{\frac{1}{3}}}
+\delta_t^2p^0\frac{t_1-t}{t_1-t_{\frac{1}{3}}}.
\end{equation}

Consider the approximation of the function $p$ defined on the interval $[t_{l-\frac{1}{2}},t_{l+\frac{1}{2}}]\,(1\leq l\leq k-1).$ Utilizing
the data $(t_{l-1},p^{l-1}), (t_{l},p^{l}),$ $(t_{l+1},p^{l+1}), (t_{l+2},p^{l+2})$ yields a cubic Newton interpolation polynomial ($\mathrm{N3}$) as follows,
\[N_{3,l}(t)=p^{l-1}+\delta_tp^{l-\frac{1}{2}}(t-t_{l-1})
+\frac{1}{2}\delta_t^2p^{l}(t-t_{l-1})(t-t_{l})
+\frac{1}{6\tau}(\delta_t^2p^{l+1}-\delta_t^2p^{l})
(t-t_{l-1})(t-t_{l})(t-t_{l+1}).\]
Taking the second order derivative on both hand sides gives
\begin{equation}\label{Newt}
N''_{3,l}(t)=\delta_t^2p^{l+1}\,\frac{t-t_{l}}{\tau}
+\delta_t^2p^{l}\,\frac{t_{l+1}-t}{\tau}.
\end{equation}

Let $\sigma=1-\alpha/2$ and $t_{k+\sigma}=t_k+\sigma\tau.$ Consider the approximation of the function $p$ defined on the interval $[t_{k-\frac{1}{2}},t_{k+\sigma}].$ Utilizing the data $(t_{k-1},p^{k-1}), (t_{k},p^{k}),$ $(t_{k+1},p^{k+1})$ yields a square Newton interpolation polynomial ($\mathrm{N2_\sigma}$) as follows,
\[N_{2,k}(t)=p^{k-1}+\delta_tp^{k-\frac{1}{2}}(t-t_{k-1})
+\frac{1}{2}\delta_t^2p^{k}(t-t_{k-1})(t-t_{k}),\]
which yields
\begin{equation}\label{super1}
N''_{2,k}(t)=\delta_t^2p^{k}.
\end{equation}

Now, we consider an H3N3-2$_\sigma$ approximation to Caputo derivative \eqref{caputo-m2}. Using \eqref{H3}-\eqref{super1}, we have
\begin{align}\label{appro1}
& _C\mathrm{D}_{0,t}^\alpha\, p(t)|_{t=t_{k+\sigma}}=\frac{1}{\Gamma(2-\alpha)}\int_{t_0}^{t_{k+\sigma}}
\frac{p''(s)}{(t_{k+\sigma}-s)^{\alpha-1}}\mathrm{d}s
\notag\\
\approx&\frac{1}{\Gamma(2-\alpha)}\Big[\int_{t_0}^{t_\frac{1}{2}}
\frac{H_{3,0}''(s)}{(t_{k+\sigma}-s)^{\alpha-1}}\mathrm{d}s
+\sum_{l=1}^{k-1}\int_{t_{l-\frac{1}{2}}}^{t_{l+\frac{1}{2}}}
\frac{N_{3,l}''(s)}{(t_{k+\sigma}-s)^{\alpha-1}}\mathrm{d}s+
\int_{t_{k-\frac{1}{2}}}^{t_{k+\sigma}}
\frac{N''_{2,k}(s)}{(t_{k+\sigma}-s)^{\alpha-1}}\mathrm{d}s\Big]
\notag\\
=&\frac{1}{\Gamma(2-\alpha)}\Bigg\{\int_{t_0}^{t_\frac{1}{2}}
\frac{t_1-s}{t_1-t_{\frac{1}{3}}}(t_{k+\sigma}-s)^{1-\alpha}
\mathrm{d}s\cdot\delta_t^2p^0
\notag\\
&+\Big[\int_{t_0}^{t_\frac{1}{2}}\frac{s-t_{\frac{1}{3}}}
{t_1-t_{\frac{1}{3}}}(t_{k+\sigma}-s)^{1-\alpha}\mathrm{d}s
+\int_{t_\frac{1}{2}}^{t_\frac{3}{2}}\frac{t_2-s}{\tau}
(t_{k+\sigma}-s)^{1-\alpha}\mathrm{d}s\Big]\cdot\delta_t^2p^1
\notag\\
&+\sum_{l=2}^{k-1}\Big[\int_{t_{l-\frac{3}{2}}}^
{t_{l-\frac{1}{2}}}\frac{s-t_{l-1}}
{\tau}(t_{k+\sigma}-s)^{1-\alpha}\mathrm{d}s+\int_{t_{l-\frac{1}{2}}}
^{t_{l+\frac{1}{2}}}
\frac{t_{l+1}-s}{\tau}(t_{k+\sigma}-s)^{1-\alpha}\mathrm{d}s\Big]
\cdot\delta_t^2p^l
\notag\\
&+\Big[\int_{t_{k-\frac{3}{2}}}^{t_{k-\frac{1}{2}}}
\frac{s-t_{k-1}}{\tau}(t_{k+\sigma}-s)^{1-\alpha}\mathrm{d}s
+\int_{t_{k-\frac{1}{2}}}^{t_{k+\sigma}}
(t_{k+\sigma}-s)^{1-\alpha}\mathrm{d}s\Big]\cdot\delta_t^2p^{k}\Bigg\}
\notag\\
=&\frac{1}{\Gamma(2-\alpha)}\Big[\int_{t_0}^{t_\frac{1}{2}}\frac{t_1-s}
{t_1-t_{\frac{1}{3}}}(t_{k+\sigma}-s)^{1-\alpha}\mathrm{d}s\cdot\delta_t^2p^0
+\sum_{l=1}^k(a_l^{(k,\alpha)}+b_l^{(k,\alpha)})\delta_t^2p^l\Big]
\notag\\
=&\frac{1}{\Gamma(2-\alpha)}\sum_{l=0}^kc_{k-l}^{(k,\alpha)}\delta_t^2p^l
=\frac{1}{\Gamma(2-\alpha)}\sum_{l=0}^kc_{l}^{(k,\alpha)}\delta_t^2p^{k-l}\notag\\
\equiv& \mathcal{D}p(t)|_{t=t_{k+\sigma}},\quad 1\leq k\leq N-1,
\end{align}
where
\begin{equation*} 
a_{l}^{(k,\alpha)}=
\begin{cases}
\displaystyle\int_{t_0}^{t_\frac{1}{2}}\frac{s-t_{\frac{1}{3}}}
{t_1-t_{\frac{1}{3}}}(t_{k+\sigma}-s)^{1-\alpha}\mathrm{d}s
=\frac{3}{2}\tau^{2-\alpha}\int_0^{\frac{1}{2}}(s-\frac{1}{3})
(k+\sigma-s)^{1-\alpha}\mathrm{d}s,\quad &l=1, \\[0.1in]
\displaystyle\int_{t_{l-\frac{3}{2}}}^{t_{l-\frac{1}{2}}}\frac{s-t_{l-1}}
{\tau}(t_{k+\sigma}-s)^{1-\alpha}\mathrm{d}s
=\tau^{2-\alpha}\int_0^1(s-\frac{1}{2})(k+\sigma-l
+\frac{3}{2}-s)^{1-\alpha}\mathrm{d}s,\quad &2\leq l\leq k,
\end{cases}
\end{equation*}
\begin{equation*} 
b_{l}^{(k,\alpha)}=
\begin{cases}
\displaystyle\int_{t_{l-\frac{1}{2}}}^{t_{l+\frac{1}{2}}}
\frac{t_{l+1}-s}{\tau}(t_{k+\sigma}-s)^{1-\alpha}\mathrm{d}s
=\tau^{2-\alpha}\int_0^1(\frac{3}{2}-s)(k+\sigma-l+
\frac{1}{2}-s)^{1-\alpha}\mathrm{d}s,\quad &1\leq l\leq k-1, \\
\displaystyle\int_{t_{k-\frac{1}{2}}}^{t_{k+\sigma}}
(t_{k+\sigma}-s)^{1-\alpha}\mathrm{d}s
=\tau^{2-\alpha}\int_0^{\sigma+\frac{1}{2}}s^{1-\alpha}\mathrm{d}s,\quad&l=k,
\end{cases}
\end{equation*}
and,
\begin{equation} \label{c-coe}
c_{l}^{(k,\alpha)}=
\begin{cases}
\ds a_{k-l}^{(k,\alpha)}+b_{k-l}^{(k,\alpha)},\quad 0\leq l\leq k-1, \\[0.1in]
\ds \int_{t_0}^{t_\frac{1}{2}}\frac{t_1-s}{t_1-t_{\frac{1}{3}}}
(t_{k+\sigma}-s)^{1-\alpha}\mathrm{d}s
=\frac{3}{2}\tau^{2-\alpha}\int_0^{\frac{1}{2}}(1-s)
(k+\sigma-s)^{1-\alpha}\mathrm{d}s,\quad l=k.
\end{cases}
\end{equation}
By tedious calculations, one gets
 \begin{equation*} 
a_{l}^{(k,\alpha)}=
\begin{cases}
\ds \frac{3}{2}\tau^{2-\alpha}\Bigg\{\frac{1}{(2-\alpha)
(3-\alpha)}\Big[(k+\sigma)^{3-\alpha}
-(k+\sigma-\frac{1}{2})^{3-\alpha}\Big]\\
\ds -\frac{1}{3(2-\alpha)}(k+\sigma)^{2-\alpha}
-\frac{1}{6(2-\alpha)}(k+\sigma-\frac{1}{2})^{2-\alpha}\Bigg\},\quad &l=1, \\[0.15in]
\ds \tau^{2-\alpha}\Bigg\{\frac{1}{(2-\alpha)(3-\alpha)}
\Big[(k+\sigma-l+\frac{3}{2})^{3-\alpha}
-(k+\sigma-l+\frac{1}{2})^{3-\alpha}\Big]\\
\ds -\frac{1}{2(2-\alpha)}(k+\sigma-l+\frac{3}{2})^{2-\alpha}
-\frac{1}{2(2-\alpha)}(k+\sigma-l+\frac{1}{2})^{2-\alpha}\Bigg\},
\quad & 2\leq l\leq k,
\end{cases}
\end{equation*}

\begin{equation*} 
b_{l}^{(k,\alpha)}=\begin{cases}
\ds \tau^{2-\alpha}\Bigg\{-\frac{1}{(2-\alpha)(3-\alpha)}
\Big[(k+\sigma-l+\frac{1}{2})^{3-\alpha}
-(k+\sigma-l-\frac{1}{2})^{3-\alpha}\Big]\\[0.15in]
\ds +\frac{3}{2(2-\alpha)}(k+\sigma-l+\frac{1}{2})^{2-\alpha}
-\frac{1}{2(2-\alpha)}(k+\sigma-l-\frac{1}{2})^{2-\alpha}\Bigg\},
\quad & 1\leq l\leq k-1, \\[0.15in]
\ds \frac{\tau^{2-\alpha}}{2-\alpha}(\sigma+\frac{1}{2})^{2-\alpha},\quad  &l=k,
\end{cases}\end{equation*}
and,
\begin{align*}
 c_k^{(k,\alpha)}=&\frac{3}{2}\tau^{2-\alpha}\Bigg
 \{-\frac{1}{(2-\alpha)(3-\alpha)}\Big[(k+\sigma)^{3-\alpha}
-(k+\sigma-\frac{1}{2})^{3-\alpha}\Big]\\
&+\frac{1}{2-\alpha}(k+\sigma)^{2-\alpha}
-\frac{1}{2(2-\alpha)}(k+\sigma-\frac{1}{2})^{2-\alpha}\Bigg\}.
\end{align*}
For the truncation error of H3N3-2$_\sigma$ formula $\mathcal{D}p(t)|_{t=t_{k+\sigma}},$ we have the estimation below.
\begin{theorem}\label{error}
Suppose $p\in C^{4}[t_0,t_{k+1}].$ Denote
\[r_{k}=\,_C\mathrm{D}_{0,t}^\alpha\, p(t)|_{t=t_{k+\sigma}}-\mathcal{D}p(t)|_{t=t_{k+\sigma}},
\quad k=1,2,3,\cdots,N-1.\]
Then, we have
\[|r_k|\leq\frac{1}{\Gamma(3-\alpha)}\max_{t_0\leq t\leq t_{k+1}}|p^{(4)}(t)|\big[\frac{5t_{k+\sigma}^{2-\alpha}}{6}\tau^2
-\frac{5}{8}(\sigma+\frac{1}{2})^{2-\alpha}\tau^{4-\alpha}\big].\]
\end{theorem}
\begin{proof}
Denote
\[D_{2,1}p(t)=p''(t_1)\frac{t-t_{\frac{1}{3}}}{t_1-t_{\frac{1}{3}}}
+p''(t_{\frac{1}{3}})\frac{t_1-t}{t_1-t_{\frac{1}{3}}},\quad t\in[t_{0},t_\frac{1}{2}].\]
By Taylor's expansion, we obtain
\begin{equation}\label{D21}
|p''(t)-D_{2,1}p(t)|\leq\frac{1}{3}\tau^2\max_{t_{0}\leq t\leq t_1}|p^{(4)}(t)|,\quad t\in[t_{0},t_\frac{1}{2}].
\end{equation}
On the other hand,
\[D_{2,1}p(t)-H''_{3,0}(t)=(p''(t_1)-\delta_t^2p^1)\frac{t-t_{\frac{1}{3}}}
{t_1-t_{\frac{1}{3}}}+(p''(t_{\frac{1}{3}})-\delta_t^2p^0)\frac{t_1-t}
{t_1-t_{\frac{1}{3}}},\quad t\in[t_{0},t_\frac{1}{2}].\]
Noticing
\[|p''(t_1)-\delta_t^2p^1|\leq\frac{\tau^2}{12}\max_{t_0\leq t\leq t_2}|p^{(4)}(t)|,\quad|p''(t_{\frac{1}{3}})-\delta_t^2p^0|\leq\frac{29\tau^2}{972}
\max_{t_0\leq t\leq t_1}|p^{(4)}(t)|,\]
one has
\begin{equation}\label{D21-1}
|D_{2,1}p(t)-H''_{3,0}(t)|\leq|p''(t_1)-\delta_t^2p^1|+\frac{3}{2}|p''(t_{\frac{1}{3}})
-\delta_t^2p^0|\leq\frac{83}{648}\tau^2\max_{t_0\leq t\leq t_2}|p^{(4)}(t)|.
\end{equation}
Combination of \eqref{D21} and \eqref{D21-1} leads to
\begin{align}\label{D21-2}
|p''(t)-H''_{3,0}(t)|\leq|p''(t)-D_{2,1}p(t)|+|D_{2,1}p(t)-H''_{3,0}(t)|\leq
\frac{299}{648}\tau^2\max_{t_0\leq t\leq t_2}|p^{(4)}(t)|.
\end{align}

Let
\[D_{2,l}p(t)=p''(t_{l+1})\frac{t-t_{l}}{\tau}+p''(t_{l})\frac{t_{l+1}-t}{\tau},\quad t\in[t_{l-\frac{1}{2}},t_{l+\frac{1}{2}}],\quad 1\leq l\leq k-1.\]
By Taylor's expansion, we have
\begin{equation}\label{D2k}
|p''(t)-D_{2,l}p(t)|\leq\frac{3}{4}\tau^2\max_{t_{l-\frac{1}{2}}\leq t\leq t_{l+1}}|p^{(4)}(t)|,\quad t\in[t_{l-\frac{1}{2}},t_{l+\frac{1}{2}}],\quad 1\leq l\leq k-1.
\end{equation}
On the other side,
\[D_{2,l}p(t)-N''_{3,l}(t)=(p''(t_{l+1})-\delta_t^2p^{l+1})\frac{t-t_{l}}{\tau}
+(p''(t_{l})-\delta_t^2p^{l})\frac{t_{l+1}-t}{\tau},\quad t\in[t_{l-\frac{1}{2}},t_{l+\frac{1}{2}}],\quad 1\leq l\leq k-1.\]
Note that for any fixed $j=l,l+1,$
\[|p''(t_j)-\delta_t^2p^j|\leq\frac{\tau^2}{12}\max_{t_{j-1}\leq t\leq t_{j+1}}|p^{(4)}(t)|.\]
Thus we obtain
\begin{equation}\label{D22-1}
|D_{2,l}p(t)-N''_{3,l}(t)|\leq\frac{\tau^2}{12}\max_{t_{l-1}\leq t\leq t_{l+2}}|p^{(4)}(t)|.
\end{equation}
Combining \eqref{D2k} and \eqref{D22-1} arrives at
\begin{align}\label{D22-2}
&|p''(t)-N''_{3,l}(t)|\leq|p''(t)-D_{2,l}p(t)|+|D_{2,l}p(t)-N''_{3,l}(t)|
\leq\frac{5}{6}\tau^2\max_{t_{l-1}\leq t\leq t_{l+2}}|p^{(4)}(t)|.
\end{align}
By the Taylor expansion again, one has
\[p''(s)=p''(t_k)+(s-t_k)p'''(t_k)+\frac{1}{2}(s-t_k)^2p^{(4)}(\eta_k),\quad \eta_k\in(t_{k-\frac{1}{2}},t_{k+\sigma}).\]
Thus we have
\begin{align*}
\mathrm{I}=&\int_{t_{k-\frac{1}{2}}}^{t_{k+\sigma}}\frac{p''(s)-N_{2,k}''(s)}
{(t_{k+\sigma}-s)^{\alpha-1}}\mathrm{d}s\notag
\\
=&\int_{t_{k-\frac{1}{2}}}^{t_{k+\sigma}}\frac{p''(t_k)-\delta_t^2p^k+(s-t_k)p'''(t_k)
+\frac{1}{2}(s-t_k)^2p^{(4)}(\eta_k)}{(t_{k+\sigma}-s)^{\alpha-1}}\mathrm{d}s\notag
\\
=&\int_{t_{k-\frac{1}{2}}}^{t_{k+\sigma}}\frac{(s-t_k)p'''(t_k)
-\frac{\tau^2}{12}p^{(4)}(\hat{\eta}_k)
+\frac{1}{2}(s-t_k)^2p^{(4)}(\eta_k)}{(t_{k+\sigma}-s)^{\alpha-1}}\mathrm{d}s,\quad \hat{\eta}_k\in(t_{k-1},t_{k+1}).
\end{align*}
Using the variable substitution, one gets
\begin{align*}
&\int_{t_{k-\frac{1}{2}}}^{t_{k+\sigma}}\frac{s-t_k}
{(t_{k+\sigma}-s)^{\alpha-1}}\mathrm{d}s
=\tau^{3-\alpha}\int_0^{\sigma+\frac{1}{2}}\frac{\sigma-s}{s^{\alpha-1}}\mathrm{d}s\\
=&\frac{1}{(2-\alpha)(3-\alpha)}\tau^{3-\alpha}(\sigma+\frac{1}{2})^{2-\alpha}\Big[\sigma-(1-\frac\alpha2)\Big]=0,
\end{align*}
where $\sigma=1-\alpha/2$ is used. Therefore,
\begin{align}\label{I-final}
|\mathrm{I}|=&\Big|\int_{t_{k-\frac{1}{2}}}^{t_{k+\sigma}}\frac{
-\frac{\tau^2}{12}p^{(4)}(\hat{\eta}_k)
+\frac{1}{2}(s-t_k)^2p^{(4)}(\eta_k)}{(t_{k+\sigma}-s)^{\alpha-1}}\mathrm{d}s\Big|\notag\\
\leq&\max_{t_{k-1}\leq s\leq t_{k+1}}|p^{(4)}(s)|\int_{t_{k-\frac{1}{2}}}^{t_{k+\sigma}}
\frac{\frac{\tau^2}{12}+\frac{1}{2}(s-t_k)^2}{(t_{k+\sigma}-s)^{\alpha-1}}\mathrm{d}s
\leq\frac{5(\sigma+\frac{1}{2})^{2-\alpha}}{24(2-\alpha)}\max_{t_{k-1}\leq s\leq t_{k+1}}|p^{(4)}(s)|\tau^{4-\alpha}.
\end{align}
Consequently, it follows from \eqref{D21-2}, \eqref{D22-2} and \eqref{I-final} that
\begin{align}\label{r_k}
 |r_k|=& \frac{1}{\Gamma(2-\alpha)}\Big|\int_{t_0}^{t_\frac{1}{2}}\frac{p''(s)-H''_{3,0}(s)}
{(t_{k+\sigma}-s)^{\alpha-1}}\mathrm{d}s+\sum_{l=1}^{k-1}
\int_{t_{l-\frac{1}{2}}}^{t_{l+\frac{1}{2}}}\frac{p''(s)-N''_{3,l}(s)}
{(t_{k+\sigma}-s)^{\alpha-1}}\mathrm{d}s+\int_{t_{k-\frac{1}{2}}}^{t_{k+\sigma}}\frac{p''(s)-N''_{2,k}(s)}
{(t_{k+\sigma}-s)^{\alpha-1}}\mathrm{d}s\Big|\notag\\
\leq&\frac{1}{\Gamma(2-\alpha)}\Big[\frac{299}{648}\max_{t_0\leq t\leq t_2}|p^{(4)}(t)|\tau^2\int_{t_0}^{t_\frac{1}{2}}
(t_{k+\sigma}-s)^{1-\alpha}\mathrm{d}s\notag\\
&+\frac{5}{6}\tau^2\sum_{l=1}^{k-1}
\int_{t_{l-\frac{1}{2}}}^{t_{l+\frac{1}{2}}}\max_{t_{l-1}\leq t\leq t_{l+2}}|p^{(4)}(t)|
(t_{k+\sigma}-s)^{1-\alpha}\mathrm{d}s+\frac{5(\sigma+\frac12)^{2-\alpha}}{24\Gamma(2-\alpha)}\max_{t_{k-1}\leq t\leq t_{k+1}}|p^{(4)}(t)|\tau^{4-\alpha}\Big]\notag\\
\leq&\frac{1}{\Gamma(2-\alpha)}\Big[\frac{5}{6}\tau^2\max_{t_{0}\leq t\leq t_{k+1}}|p^{(4)}(t)|
\int_{t_0}^{t_{k-\frac{1}{2}}}(t_{k+\sigma}-s)^{1-\alpha}\mathrm{d}s+
\frac{5(\sigma+\frac12)^{2-\alpha}}{24\Gamma(2-\alpha)}\max_{t_{0}\leq t\leq t_{k+1}}|p^{(4)}(t)|\tau^{4-\alpha}\Big]\notag\\
=&\frac{5}{6\Gamma(3-\alpha)}\max_{t_0\leq t\leq t_{k+1}}|p^{(4)}(t)|\tau^2\left[t_{k+\sigma}^{2-\alpha}-(\sigma+\frac12)^{2-\alpha}\tau^{2-\alpha}\right]+
\frac{5(\sigma+\frac12)^{2-\alpha}}{24\Gamma(3-\alpha)}\max_{t_{0}\leq t\leq t_{k+1}}|p^{(4)}(t)|\tau^{4-\alpha}\notag\\
=&\frac{1}{\Gamma(3-\alpha)}\max_{t_0\leq t\leq t_{k+1}}|p^{(4)}(t)|\big[\frac{5t_{k+\sigma}^{2-\alpha}}{6}\tau^2
-\frac{5}{8}(\sigma+\frac{1}{2})^{2-\alpha}\tau^{4-\alpha}\big],
\end{align}
when $\sigma=1-\alpha/2.$ This completes the proof.
\end{proof}
\begin{remark} Here, we take $\sigma=1-\alpha/2$ mainly to make the following equation holds
\[\int_{t_{k-\frac{1}{2}}}^{t_{k+\sigma}}\frac{s-t_k}
{(t_{k+\sigma}-s)^{\alpha-1}}\mathrm{d}s=\frac{1}{(2-\alpha)(3-\alpha)}
\tau^{3-\alpha}(\sigma+\frac{1}{2})^{2-\alpha}\Big[\sigma-(1-\frac\alpha2)\Big]=0.\]
Otherwise, the above equation has an estimate $O(\tau^{3-\alpha}),$ which leads to the truncation error $r_k=O(\tau^{3-\alpha}),$ so that the numerical accuracy does not reach the overall second-order accuracy.
\end{remark}
\begin{lemma}\label{mono}
For $\{c_l^{(k,\alpha)}\,|\,0\leq l\leq k, k\geq1\}$ defined in \eqref{c-coe}, it holds that
\begin{align}
& c_0^{(k,\alpha)}>c_1^{(k,\alpha)}>\cdots>
c_{k-1}^{(k,\alpha)}>c_k^{(k,\alpha)}>
\frac{3(k+\sigma)^{1-\alpha}}{8}\tau^{2-\alpha}>0,
\label{coeff}\\[0.1in]
& 4\sigma c_0^{(k,\alpha)}-(1+2\sigma)c_1^{(k,\alpha)}>0,\label{positive-1}\\[0.1in]
&\sum_{m=1}^kc_{m-1}^{(m,\alpha)}<\frac{(8\sigma+21)T^{2-\alpha}}
{16(1+2\sigma)\sigma},\label{sum}
\end{align}
\begin{align}\label{sum1}
 \sum_{m=1}^k\big(c_m^{(m,\alpha)}\big)^2
\leq\ds\frac{9C}{16(1+2\sigma)^2}
\ds\begin{cases}
\ds \frac{(1+\sigma)^{2}T^{3-2\alpha}}{3-2\alpha}\tau+(1+\sigma)^{4-2\alpha},\quad & \alpha\in(1,1.5),\\
\ds (1+\sigma)T,\quad & \alpha=1.5,\\
\ds \frac{2-3\sigma}{2\alpha-3}(1+\sigma)^{4-2\alpha}\tau^{4-2\alpha},\quad &\alpha\in(1.5,2),
\end{cases}\end{align}
where $C$ is a positive constant, independent of $\tau$ and $k.$
\end{lemma}
\begin{proof} We firstly prove \eqref{coeff}. For $k=1,$ noticing $\sigma=1-\alpha/2,$ one has
\begin{align*}
&c_0^{(1,\alpha)}-c_1^{(1,\alpha)}=a_1^{(1,\alpha)}+b_1^{(1,\alpha)}-c_1^{(1,\alpha)}
\\
=&\frac{3\tau^{2-\alpha}}{2(2-\alpha)}\Big\{\frac{1}{3-\alpha}
\big[(1+\sigma)^{3-\alpha}-(\frac{1}{2}+\sigma)^{3-\alpha}\big]-\frac{1}{3}(1+\sigma)^{2-\alpha}
-\frac{1}{6}(\frac{1}{2}+\sigma)^{2-\alpha}
\Big\}
\\
&+\frac{\tau^{2-\alpha}}{2-\alpha}(\frac{1}{2}+\sigma)^{2-\alpha}+
\frac{3\tau^{2-\alpha}}{2(2-\alpha)}\Big\{\frac{1}{3-\alpha}
\big[(1+\sigma)^{3-\alpha}-(\frac{1}{2}+\sigma)^{3-\alpha}\big]-(1+\sigma)^{2-\alpha}
+\frac{1}{2}(\frac{1}{2}+\sigma)^{2-\alpha}
\Big\}
\\
=&\frac{3\tau^{2-\alpha}}{(2-\alpha)(3-\alpha)}
\big[(1+\sigma)^{3-\alpha}-(\frac{1}{2}+\sigma)^{3-\alpha}\big]
-\frac{2\tau^{2-\alpha}}{2-\alpha}(1+\sigma)^{2-\alpha}
+\frac{3\tau^{2-\alpha}}{2(2-\alpha)}(\frac{1}{2}+\sigma)^{2-\alpha}
\\
=&\frac{\tau^{2-\alpha}}{(2-\alpha)(3-\alpha)}(1+\sigma)^{2-\alpha}
\big[3\sigma+2\alpha-3\big]+\frac{2\tau^{2-\alpha}}{(2-\alpha)(3-\alpha)}(\frac{1}{2}+\sigma)^{2-\alpha}
\big[1-\frac{\alpha}{2}-\sigma\big]
\\
=&\frac{(1-\sigma)(1+\sigma)^{2-\alpha}}{(2-\alpha)(3-\alpha)}\tau^{2-\alpha}>0.
\end{align*}

For $2\leq l\leq k\,(k\geq 2),$ one has
\begin{align*}
a_l^{(k,\alpha)}
=&\tau^{2-\alpha}\int_0^1(s-\frac{1}{2})
(k+\sigma-l+\frac{3}{2}-s)^{1-\alpha}\mathrm{d}s\notag\\
=&\tau^{2-\alpha}\Big[\int_0^\frac{1}{2}(s-\frac{1}{2})(k+\sigma-l+\frac{3}{2}-s)^{1-\alpha}\mathrm{d}s
+\int_\frac{1}{2}^1(s-\frac{1}{2})(k+\sigma-l+\frac{3}{2}-s)^{1-\alpha}\mathrm{d}s\Big]\\
=&\tau^{2-\alpha}\int_0^{\frac{1}{2}}s\Big[(k+\sigma-l+1-s)^{1-\alpha}
-(k+\sigma-l+1+s)^{1-\alpha}\Big]\mathrm{d}s
=\tau^{2-\alpha}\int_0^\frac{1}{2}sE(l,s)\mathrm{d}s,
\end{align*}
where we used variable substitution $s-\frac{1}{2}=\pm t$ in the second equality and $E(l,s)=(k+\sigma-l+1-s)^{1-\alpha}-(k+\sigma-l+1+s)^{1-\alpha}.$ \,It is easy to verify that $\frac{\partial E(l,s)}{\partial l}>0.$ Therefore, $E(l,s)$ rigorously increases with respect to $l,$ so does $a_l^{(k,\alpha)}.$

For $1\leq l\leq k-1\,(k\geq2),$ we get
\begin{align*}
b_l^{(k,\alpha)}=\tau^{2-\alpha}\int_0^1(s+\frac{1}{2})
(k+\sigma-l-\frac{1}{2}+s)^{1-\alpha}\mathrm{d}s.
\end{align*}
It is easy to know that $b_l^{(k,\alpha)}$ rigorously increases with respect to $l.$

Besides, for $k\geq2,$ we have
\begin{align*}
&a_1^{(k,\alpha)}-a_{2}^{(k,\alpha)}
\\
=&\tau^{2-\alpha}\Big[\int_0^\frac{1}{2}(\frac{3s}{2}
-\frac{1}{2})(k+\sigma-s)^{1-\alpha}\mathrm{d}s
-\int_0^1(s-\frac{1}{2})(k+\sigma-\frac{1}{2}-s)^{1-\alpha}\mathrm{d}s\Big]
\\
=&\tau^{2-\alpha}\Big[\int_\frac{1}{3}^\frac{1}{2}(\frac{3s}{2}
-\frac{1}{2})(k+\sigma-s)^{1-\alpha}\mathrm{d}s-\int_0^\frac{1}{3}(\frac{1}{2}-\frac{3s}{2})
(k+\sigma-s)^{1-\alpha}\mathrm{d}s
\\
&+\int_0^\frac{1}{2}(\frac{1}{2}-s)(k+\sigma-\frac{1}{2}-s)^{1-\alpha}\mathrm{d}s
-\int_\frac{1}{2}^1(s-\frac{1}{2})(k+\sigma-\frac{1}{2}-s)^{1-\alpha}\mathrm{d}s\Big]
\\
=&\tau^{2-\alpha}\Big[\int_0^\frac{1}{4}\frac{2s}{3}
(k+\sigma-\frac{1}{3}-\frac{2}{3}s)^{1-\alpha}\mathrm{d}s
-\int_0^\frac{1}{2}\frac{2s}{3}(k+\sigma-\frac{1}{3}
+\frac{2s}{3})^{1-\alpha}\mathrm{d}s
\\
&+\int_0^\frac{1}{2}s(k+\sigma-1+s)^{1-\alpha}\mathrm{d}s
-\int_0^\frac{1}{2}s(k+\sigma-1-s)^{1-\alpha}\mathrm{d}s\Big]
\\
<&\tau^{2-\alpha}\Big[\int_0^\frac{1}{4}\frac{2s}{3}
(k+\sigma-\frac{1}{3}-\frac{2}{3}s)^{1-\alpha}\mathrm{d}s
-\int_0^\frac{1}{4}\frac{4s}{3}(k+\sigma-\frac{1}{3}
+\frac{2s}{3})^{1-\alpha}\mathrm{d}s\Big]
\\
=&\tau^{2-\alpha}\int_0^\frac{1}{4}\frac{2s}{3}\Big[(k+\sigma-\frac{1}{3}
-\frac{2}{3}s)^{1-\alpha}
-\big[2^{(1-\alpha)^{-1}}(k+\sigma-\frac{1}{3}
+\frac{2}{3}s)\big]^{1-\alpha}\Big]\mathrm{d}s
\\
<&\tau^{2-\alpha}\int_0^\frac{1}{4}\frac{2s}{3}\Big[(k+\sigma-1)^{1-\alpha}
-\big[\frac{1}{2}(k+\sigma)\big]^{1-\alpha}\Big]\mathrm{d}s<0,
\end{align*}
here, in the second equality, we used variable substitutions: (1) $\frac{3}{2}s-\frac{1}{2}=t;$ (2) $\frac{1}{2}-\frac{3}{2}s=t;$ (3) $\frac{1}{2}-s=t;$ (4) $s-\frac{1}{2}=t,$ respectively. And
\begin{align*}
&b_k^{(k,\alpha)}-b_{k-1}^{(k,\alpha)}
=\frac{\tau^{2-\alpha}}{2-\alpha}\Big[(\sigma+\frac{1}{2})^{2-\alpha}
-\frac{2\sigma}{1+2\sigma}(\sigma+\frac{3}{2})^{2-\alpha}\Big]
\\
=&\frac{\tau^{2-\alpha}}{2-\alpha}(\sigma+\frac{1}{2})^{2-\alpha}
\Big[1-\frac{2\sigma}{1+2\sigma}\big(\frac{\sigma+3/2}{\sigma+1/2}\big)^{2-\alpha}
\Big]
\\
=&\frac{\tau^{2-\alpha}}{2-\alpha}(\sigma+\frac{1}{2})^{2-\alpha}
\Big[1-\frac{2}{\frac{1}{\sigma}+2}\big(1+\frac{2}{2\sigma+1}\big)^{2-\alpha}
\Big]
\\
>&\frac{\tau^{2-\alpha}}{2-\alpha}(\sigma+\frac{1}{2})^{2-\alpha}
\Big[1-\frac{1}{2}(1+\frac{2}{2\sigma+1})^{2-\alpha}
\Big]
\\
>&\frac{\tau^{2-\alpha}}{2-\alpha}(\sigma+\frac{1}{2})^{2-\alpha}
\Big[1-\frac{1}{2}(1+\frac{2(2-\alpha)}{2\sigma+1})
\Big]>0,
\end{align*}
in which $\sigma=1-\alpha/2\in(0,1/2),$ thus $1/\sigma>2.$ So one has
\begin{align*}
&c_{k-1}^{(k,\alpha)}-c_{k}^{(k,\alpha)}=a_1^{(k,\alpha)}
+b_1^{(k,\alpha)}-c_k^{(k,\alpha)}
\\
=&\tau^{2-\alpha}\Big[\int_0^1(\frac{3}{2}-s)(k+\sigma-s-
\frac{1}{2})^{1-\alpha}\mathrm{d}s
-\int_0^1(1-\frac{3}{4}s)(k+\sigma-\frac{s}{2})^{1-\alpha}\mathrm{d}s\Big]
\\
>&\tau^{2-\alpha}\int_0^1(\frac{1}{2}-\frac{1}{4}s)
(k+\sigma-\frac{s}{2})^{1-\alpha}\mathrm{d}s
>0.
\end{align*}
Thus, we have shown that $\{c_l^{(k,\alpha)}|_{l=0}^k\}$ is rigorously decreasing with respect to the subscript $l.$ We also get that
\begin{align*}
c_k^{(k,\alpha)}
=\frac{3\tau^{2-\alpha}}{4}\int_0^1(1-\frac{s}{2})
(k+\sigma-\frac{s}{2})^{1-\alpha}\mathrm{d}s
>\frac{3(k+\sigma)^{1-\alpha}}{8}\tau^{2-\alpha}.
\end{align*}

Now we turn to prove \eqref{positive-1}. It is easy to know that $4\sigma c_0^{(1,\alpha)}-(1+2\sigma)c_1^{(1,\alpha)}>0.$  For $k\geq 2,$ we arrive at
\begin{align*}
&\frac{2-\alpha}{\tau^{2-\alpha}}\Big[4\sigma c_0^{(k,\alpha)}-(1+2\sigma)c_1^{(k,\alpha)}\Big]
=\frac{8\sigma+2}{1+2\sigma}(\sigma+\frac{3}{2})^{2-\alpha}
-2(\sigma+\frac{5}{2})^{2-\alpha}\\
=&\frac{8\sigma+2}{1+2\sigma}(\sigma+\frac{3}{2})^{2-\alpha}
-2(\sigma+\frac{3}{2})^{2-\alpha}\Big(1+\frac{2}{2\sigma+3}\Big)^{2-\alpha}
\\
\geq &\frac{8\sigma+2}{1+2\sigma}(\sigma+\frac{3}{2})^{2-\alpha}
-2(\sigma+\frac{3}{2})^{2-\alpha}\Big(1+\frac{2(2-\alpha)}{2\sigma+3}\Big)
=\frac{4\sigma(1-2\sigma)}{(1+2\sigma)(2\sigma+3)}(\sigma+\frac{3}{2})^{2-\alpha}>0.
\end{align*}

Next, we prove \eqref{sum}. For $k=1,$
\begin{align*}
\frac{2-\alpha}{\tau^{2-\alpha}}c_0^{(1,\alpha)}=\frac{2-\alpha}{\tau^{2-\alpha}}
\big(a_1^{(1,\alpha)}+b_1^{(1,\alpha)}\big)=\frac{2+\sigma}
{2(1+2\sigma)}(1+\sigma)^{2-\alpha}.
\end{align*}
Thereby,
\begin{align*}
\sum_{m=1}^kc_{m-1}^{(m,\alpha)}=c_{0}^{(1,\alpha)}
=a_1^{(1,\alpha)}+b_1^{(1,\alpha)}
=\frac{(2+\sigma)\tau^{2-\alpha}}{2(1+2\sigma)(2-\alpha)}(1+\sigma)^{2-\alpha}.
\end{align*}

For $k\geq2,$
\begin{align*}
&\frac{2-\alpha}{\tau^{2-\alpha}}\big(a_1^{(k,\alpha)}+b_1^{(k,\alpha)}\big)\\
=&\frac{(k+\sigma-\frac{1}{2})^{2-\alpha}}{1+2\sigma}
\Big[\frac{3k+\sigma-1}{2}\big(1+\frac{1}{2k+2\sigma-1}\big)^{2-\alpha}
-\frac{5k-5}{2}+(k-2)\big(1-\frac{2}{2k+2\sigma-1}\big)^{2-\alpha}\Big].
\end{align*}
Consequently, we have
\begin{align*}
&\sum_{m=1}^kc_{m-1}^{(m,\alpha)}=\sum_{m=1}^k
\big(a_1^{(m,\alpha)}+b_1^{(m,\alpha)}\big)\\
=&\frac{(2+\sigma)\tau^{2-\alpha}}{2(1+2\sigma)(2-\alpha)}(1+\sigma)^{2-\alpha}+
\sum_{m=2}^k\frac{(m+\sigma-\frac{1}{2})^{2-\alpha}\tau^{2-\alpha}}{(1+2\sigma)(2-\alpha)}
\Big[\frac{3m+\sigma-1}{2}\big(1+\frac{1}{2m+2\sigma-1}\big)^{2-\alpha}
\\
&-\frac{5m-5}{2}+(m-2)\big(1-\frac{2}{2m+2\sigma-1}\big)^{2-\alpha}\Big]
\\
=&\frac{(2+\sigma)\tau^{2-\alpha}}{2(1+2\sigma)(2-\alpha)}(1+\sigma)^{2-\alpha}+
\frac{\tau^{2-\alpha}}{(1+2\sigma)(2-\alpha)}\sum_{m=2}^k
(m+\sigma-\frac{1}{2})^{2-\alpha}
P^m,
\end{align*}
where $\ds P^m=\frac{3m+\sigma-1}{2}\big(1+\frac{1}{2m+2\sigma-1}\big)^{2-\alpha}
-\frac{5m-5}{2}+(m-2)\big(1-\frac{2}{2m+2\sigma-1}\big)^{2-\alpha}.$\\[0.1in]
With the help of the inequality $(1+x)^r\leq 1+rx,\quad r\in(0,1),~1>x>-1,$ we have
\begin{align*}
P^m\leq&\frac{3m+\sigma-1}{2}\big(1+\frac{2-\alpha}{2m+2\sigma-1}\big)
-\frac{5m-5}{2}+(m-2)
\big(1-\frac{2(2-\alpha)}{2m+2\sigma-1}\big)\\
=&\frac{\sigma}{2}+\frac{(2-\alpha)(\sigma+7-m)}{2(2m+2\sigma-1)}
=\frac{\sigma(4\sigma+13)}{2(2m+2\sigma-1)}=\frac{\sigma(4\sigma+13)}{4(m+\sigma-\frac{1}{2})}.
\end{align*}
Further, we get
\begin{align*}
&\frac{\tau^{2-\alpha}}
{(1+2\sigma)(2-\alpha)}\sum_{m=2}^k(m+\sigma-\frac{1}{2})^{2-\alpha}P^m
\leq\frac{(4\sigma+13)\tau^{2-\alpha}}
{8(1+2\sigma)}\sum_{m=2}^k(m+\sigma-\frac{1}{2})^{1-\alpha}\\
\leq&\frac{(4\sigma+13)\tau^{2-\alpha}}
{8(1+2\sigma)}\sum_{m=2}^{k}\int_{m-1}^m(s+\sigma-\frac{1}{2})^{1-\alpha}\mathrm{d}s
=\frac{(4\sigma+13)\tau^{2-\alpha}}
{8(1+2\sigma)}\int_{1}^k(s+\sigma-\frac{1}{2})^{1-\alpha}\mathrm{d}s
\\
<&\frac{(4\sigma+13)\tau^{2-\alpha}}
{16(1+2\sigma)\sigma}(k+\sigma-\frac{1}{2})^{2-\alpha}
<\frac{(4\sigma+13)T^{2-\alpha}}
{16(1+2\sigma)\sigma},
\end{align*}
where we have used the mean value theorem of integrals, that is \[\int_{m-1}^m(s+\sigma-\frac{1}{2})^{1-\alpha}\mathrm{d}s=(\xi+\sigma-\frac{1}{2})^{1-\alpha}\geq
(m+\sigma-\frac{1}{2})^{1-\alpha} ,~\xi\in(m-1,m).\]
Thus, one gets
\[\sum_{m=1}^kc_{m-1}^{(m,\alpha)}<\frac{(8\sigma+21)T^{2-\alpha}}
{16(1+2\sigma)\sigma}.\]

Finally, we are going to prove \eqref{sum1}. For $k\geq1,$
\begin{align*}
c_k^{(k,\alpha)}
=\frac{3\tau^{2-\alpha}}{2(1+2\sigma)(2-\alpha)}\Big[(k-1)
\big(k+\sigma-\frac{1}{2}\big)^{2-\alpha}-(k-1-\sigma)
\big(k+\sigma\big)^{2-\alpha}\Big].
\end{align*}
Hence,
\begin{align*}
\sum_{m=1}^k\big(c_m^{(m,\alpha)}\big)^2
=\frac{9\tau^{4-2\alpha}}{4(1+2\sigma)^2(2-\alpha)^2}\sum_{m=1}^k\Big[(m-1)
\big(m+\sigma-\frac{1}{2}\big)^{2-\alpha}-(m-1-\sigma)
\big(m+\sigma\big)^{2-\alpha}\Big]^2.
\end{align*}
Since the summation
\begin{align}\label{sum-217}
&\sum_{m=1}^k\Big[(m-1)
\big(m+\sigma-\frac{1}{2}\big)^{2-\alpha}-(m-1-\sigma)
\big(m+\sigma\big)^{2-\alpha}\Big]^2\notag\\
=&\sum_{m=1}^k(m-1)^2(m+\sigma)^{4-2\alpha}\Big[\big(1-\frac{1}{2(m+\sigma)}\big)
^{2-\alpha}-\big(1-\frac{\sigma}{m-1}\big)\Big]^2\notag\\
\leq&C\cdot\sum_{m=1}^k(m-1)^2(m+\sigma)^{4-2\alpha}\Big[1-\frac{2\sigma}{2(m+\sigma)}
-\big(1-\frac{\sigma}{m-1}\big)\Big]^2\notag\\
=&C\cdot\sigma^2(1+\sigma)^2\sum_{m=1}^k\frac{1}{(m+\sigma)^{2\alpha-2}}
\leq C\cdot\sigma^2(1+\sigma)^2\Big[\frac{1}{(1+\sigma)^{2\alpha-2}}+
\sum_{m=2}^k\int_{m-1}^m\frac{1}{(s+\sigma)^{2\alpha-2}}\mathrm{d}s\Big] \\
=&C\cdot\sigma^2(1+\sigma)^2\Big[\frac{1}{(1+\sigma)^{2\alpha-2}}+
\int_{1}^k\frac{1}{(s+\sigma)^{2\alpha-2}}\mathrm{d}s\Big]\notag\\
=&C\cdot\sigma^2(1+\sigma)^2\Big[\frac{3\sigma-2}{3-2\alpha}(1+\sigma)^{2-2\alpha}+
\frac{1}{3-2\alpha}(k+\sigma)^{3-2\alpha}\Big]\notag\\
<&C\cdot\begin{cases}
\ds \frac{\sigma^2(1+\sigma)^2}{3-2\alpha}(k+\sigma)^{3-2\alpha}
+\sigma^2(1+\sigma)^{4-2\alpha},\quad  \alpha\in(1,1.5),\notag\\
\ds \sigma^2(1+\sigma)^2\Big[\frac{1}{1+\sigma}+\ln\big(1+\frac{k-1}{1+\sigma}\big)\Big]
<k\sigma^2(1+\sigma),\quad  \alpha=1.5,\notag\\
\ds \frac{2-3\sigma}{2\alpha-3}\sigma^2(1+\sigma)^{4-2\alpha},\quad \alpha\in(1.5,2),\notag
\end{cases}
\end{align}
where $C$ is a positive constant, and where the mean value theorem of integrals in \eqref{sum-217} is used, that is \[\int_{m-1}^m\frac{1}{(s+\sigma)^{2\alpha-2}}\mathrm{d}s=\frac{1}{(\xi+\sigma)^{2\alpha-2}}\geq
\frac{1}{(m+\sigma)^{2\alpha-2}} ,~\xi\in(m-1,m).\]

One has
\begin{align*}
 \sum_{m=1}^k\big(c_m^{(m,\alpha)}\big)^2
\leq\ds\frac{9C}{16(1+2\sigma)^2}
\ds\begin{cases}
\ds \frac{(1+\sigma)^{2}T^{3-2\alpha}}{3-2\alpha}\tau+(1+\sigma)^{4-2\alpha},\quad & \alpha\in(1,1.5),\\
\ds (1+\sigma)T,\quad & \alpha=1.5,\\
\ds \frac{2-3\sigma}{2\alpha-3}(1+\sigma)^{4-2\alpha}\tau^{4-2\alpha},\quad &\alpha\in(1.5,2),
\end{cases}\end{align*}
in which $C$ is a positive constant, independent of $\tau$ and $k.$
Therefore, the proof is complete.
\end{proof}
\begin{lemma}\label{le3}
Suppose $p\in C^2[0,T],$ $\sigma\in(0,\frac{1}{2}),$ then
\begin{equation*}
p(t_{k+1})=\big(\frac{3}{2}-\sigma\big)\Big[(\frac{1}{2}+\sigma)p(t_{k+1})
+(\frac{1}{2}-\sigma)p(t_{k})\Big]
-\big(\frac{1}{2}-\sigma\big)\Big[(\frac{1}{2}+\sigma)p(t_{k})
+(\frac{1}{2}-\sigma)p(t_{k-1})\Big]+O(\tau^2).
\end{equation*}
\end{lemma}
\begin{proof}
By the linear interpolation, we have
\begin{align*}
p(t_{k+1})=&\big(\frac{3}{2}-\sigma\big)p(t_{k+\sigma+\frac{1}{2}})-
\big(\frac{1}{2}-\sigma\big)p(t_{k+\sigma-\frac{1}{2}})+O(\tau^2)\\
=&\big(\frac{3}{2}-\sigma\big)\Big[(\frac{1}{2}+\sigma)p(t_{k+1})
+(\frac{1}{2}-\sigma)p(t_{k})\Big]
-\big(\frac{1}{2}-\sigma\big)\Big[(\frac{1}{2}+\sigma)p(t_{k})
+(\frac{1}{2}-\sigma)p(t_{k-1})\Big]+O(\tau^2).
\end{align*}
The proof is ended.
\end{proof}

Although the above formula holds for other $\sigma$'s values, we take $\sigma\in(0,\frac 12)$ just for later convenience.
\section{Finite difference scheme}\label{scheme}
This section devotes to presenting a second-order difference scheme for solving problem \eqref{eq}-\eqref{bz}.

We firstly introduce some spatial notations. Given a positive constant $M,$ denote $h=L/M.$ Let $x_i=ih\,(0\leq i\leq M),$ where $x_0=0,\, x_M=L,$ $\Omega_h=\{x_i\,|\,0\leq i\leq M\}.$
Let
\[\mathcal{U}_h=\{u\,|\,u=(u_0,\cdots,u_M)\},\quad \mathring{\mathcal{U}}_h=\{u\,|\,u=(u_0,\cdots,u_M),\quad u_0=u_M=0\}.\]
For $u\in\mathcal{U}_h,$ introduce the following notations
\[\delta_xu_{i+\frac{1}{2}}=\frac{1}{h}(u_{i+1}-u_i),\quad \delta_x^2u_i=\frac 1{h^2}(u_{i+1}-2u_i+u_{i-1}).\]
For any $u,v\in\mathring{\mathcal{U}}_h,$ the inner products defined by
\[(u,v)=h\sum_{i=1}^{M-1}u_iv_i,\quad \langle u,v\rangle=h\sum_{i=0}^{M-1}
(\delta_xu_{i+\frac{1}{2}})(\delta_xv_{i+\frac{1}{2}}).\]
We define some useful norms as follows:
\begin{align*}
&\|u\|=\sqrt{(u,u)}\qquad \text{discrete $L^2$-norm (mean norm);}\\
&|u|_1=\sqrt{\langle u,u\rangle}\qquad \text{discrete $H^1$-norm ($L^2$-norm of difference quotient);}\\
&\|u\|_\infty=\max_{0\leq i\leq M}|u_i|\qquad \text{discrete infinite-norm (uniform-norm).}
\end{align*}
For $u\in\mathring{\mathcal{U}}_h,$ the following two main embedded inequalities are always true \cite{sunzz-pd}
\[\|u\|_\infty\leq \frac{L}{\sqrt{6}}|u|_1,\quad \|u\|\leq\frac{\sqrt{L}}{2}|u|_1.\]

\subsection{Derivation of the difference scheme}\label{sub-sec-scheme}
Suppose $u(x,t)\in C^{(4,4)}([0,L]\times[0,T]).$ Denote
\[U_i^k=u(x_i,t_k),\,\, 0\leq k\leq N, \quad f_i^{k+\sigma}=f(x_i,t_{k+\sigma}),\,\, 0\leq k\leq N-1; \quad \varphi_i=\varphi(x_i),
\quad \psi_i=\psi(x_i),\,\, 0\leq i\leq M.\]
Considering Eq. \eqref{eq} at the point $(x_i,t_{k+\sigma}),$ we get
\[_C\mathrm{D}_{0,t}^\alpha u(x_i,t_{k+\sigma})=u_{xx}(x_i,t_{k+\sigma}) +f(x_i,t_{k+\sigma}),\quad 1\leq i\leq M-1,~1\leq k\leq N-1.\]
Apply \eqref{appro1} to approximating the temporal fractional derivative.
For the spatial derivative, one obtains
\begin{equation}\label{spatial}\begin{array}{l}
\ds \hspace{0.15in} u_{xx}(x_i,t_{k+\sigma})
\\[0.1in]
=\ds u_{xx}(x_i,t_{k})+\sigma\tau\cdot\frac{u_{xx}(x_i,t_{k+1})
-u_{xx}(x_i,t_{k-1})}{2\tau}+O(\tau^2)\\[0.1in]
=\ds (\frac{1}{2}+\sigma)\cdot\frac{u_{xx}(x_i,t_{k})
+u_{xx}(x_i,t_{k+1})}{2}+(\frac{1}{2}-\sigma)\cdot
\frac{u_{xx}(x_i,t_{k-1})+u_{xx}(x_i,t_{k})}{2}+O(\tau^2)
\\[0.1in]
=\ds (\frac{1}{2}+\sigma)\delta_x^2U_i^{k+\frac{1}{2}}
+(\frac{1}{2}-\sigma)\delta_x^2U_i^{k-\frac{1}{2}}+O(\tau^2+h^2),
\quad 1\leq i\leq M-1,\quad 1\leq k\leq N-1.
\end{array}\end{equation}
Thus, we have
\begin{equation}\label{numerical}\begin{array}{l}
\ds \hspace{0.15in} \frac{1}{\Gamma(2-\alpha)}\sum_{l=0}^kc_{l}^{(k,\alpha)}\delta_t^2U_i^{k-l}
\\[0.1in]
=\ds (\frac{1}{2}+\sigma)\delta_x^2U_i^{k+\frac{1}{2}}+(\frac{1}{2}-\sigma)
\delta_x^2U_i^{k-\frac{1}{2}}+f_i^{k+\sigma}+R_i^k,\quad 1\leq i\leq M-1,~1\leq k\leq N-1,
\end{array}\end{equation}
where
\begin{align*}\delta_t^2U_i^l
\begin{split}
=\left\{
\begin{array}{ll}
\ds \frac{1}{\tau^2}\big(U_i^{l+1}-2U_i^l+U_i^{l-1}\big),&1\leq l\leq N-1,
\\[0.1in]
\ds\frac{2}{\tau}\big[\frac{1}{\tau}(U_i^1-U_i^0)-\psi_i\big],&l=0.
\end{array}
\right.\end{split}
\end{align*}
Obviously, there exists a positive constant $c_1$ such that
\begin{equation}\label{trunction}
|R_i^k|\leq c_1(\tau^2+h^2),\quad 1\leq i\leq M-1,~ 1\leq k\leq N-1.
\end{equation}

For the computation of the approximate value at $t=t_1$, we consider
\begin{equation}\label{initial-approximate}\begin{array}{l}
\ds \hspace{0.15in} _C\mathrm{D}_{0,t}^\alpha\, p(t)|_{t=t_{1-\frac{\alpha}{3}}}=\frac{1}{\Gamma(2-\alpha)}
\int_{t_0}^{t_{1-\frac{\alpha}{3}}}\frac{p''(s)}
{(t_{1-\frac{\alpha}{3}}-s)^{\alpha-1}}
\mathrm{d}s
\\[0.15in]
\approx\ds \frac{1}{\Gamma(2-\alpha)}\int_{t_0}^{t_{1-\frac{\alpha}{3}}}
\frac{\delta_t^2p^0}{(t_{1-\frac{\alpha}{3}}-s)^{\alpha-1}}\mathrm{d}s
\\[0.15in]
=\ds\frac{1}{\Gamma(3-\alpha)}t_{1-\frac{\alpha}{3}}^{2-\alpha}\cdot\frac{2}{\tau}
(\delta_tp^{\frac{1}{2}}-p'(t_0))\equiv \widehat{\mathcal{D}}p(t)|_{t=t_{1-\frac{\alpha}{3}}}.
\end{array}\end{equation}
Noticing the fact that
\[\int_{t_0}^{t_{1-\frac{\alpha}{3}}}(s-t_{\frac{1}{3}})
(t_{1-\frac{\alpha}{3}}-s)^{1-\alpha}\mathrm{d}s=0,\]
one gets
\begin{align*}
&_C\mathrm{D}_{0,t}^\alpha\, p(t)|_{t=t_{1-\frac{\alpha}{3}}}-\widehat{\mathcal{D}}
p(t)|_{t=t_{1-\frac{\alpha}{3}}}\notag
\\
=&\frac{1}{\Gamma(2-\alpha)}\int_{t_0}^{t_{1-\frac{\alpha}{3}}}
\frac{p''(s)-\delta_t^2p^0}
{(t_{1-\frac{\alpha}{3}}-s)^{\alpha-1}}\mathrm{d}s\notag
\\
=&\frac{1}{\Gamma(2-\alpha)}\int_{t_0}^{t_{1-\frac{\alpha}{3}}}
\big[p'''(t_{\frac{1}{3}})
(s-t_{\frac{1}{3}})+\frac{1}{2}p^{(4)}(\tilde{\eta})(s-t_{\frac{1}{3}})^2
+p''(t_{\frac{1}{3}})-\delta_t^2p^0\big](t_{1-\frac{\alpha}{3}}-s)^{1-\alpha}
\mathrm{d}s\notag
\\
=&\frac{1}{\Gamma(2-\alpha)}\int_{t_0}^{t_{1-\frac{\alpha}{3}}}
\big[\frac{1}{2}p^{(4)}(\tilde{\eta})(s-t_{\frac{1}{3}})^2+
p''(t_{\frac{1}{3}})-\delta_t^2p^0\big]
(t_{1-\frac{\alpha}{3}}-s)^{1-\alpha}\mathrm{d}s\notag
\\
=&O(\tau^{4-\alpha}),\quad \tilde{\eta}\in(t_0,t_{1-\frac{\alpha}{3}}).
\end{align*}
Considering Eq. \eqref{eq} at the point $(x_i,t_{1-\frac{\alpha}{3}})$ gives
\begin{equation}\label{scheme-eq1}
\frac{1}{\Gamma(3-\alpha)}t_{1-\frac{\alpha}{3}}^{2-\alpha}\cdot\frac{2}{\tau}
(\delta_tU_i^{\frac{1}{2}}-\psi_i)=\delta_x^2\big[(1-\frac{\alpha}{3})U_i^1
+\frac{\alpha}{3}U_i^0\big]+f(x_i,t_{1-\frac{\alpha}{3}})+R_i^0,\quad 1\leq i\leq M-1,
\end{equation}
in which there is a positive constant $c_2$ such that
\begin{equation}\label{error1}
|R_i^0|\leq c_2(\tau^{4-\alpha}+h^2),\quad 1\leq i\leq M-1.
\end{equation}

Noticing the initial-boundary value conditions \eqref{cz} and \eqref{bz}, we have
\begin{equation}\label{czs}\begin{array}{l}
\ds U_i^0=\varphi_i,\quad  1\leq i\leq M-1,
\\[0.1in]
\ds U_0^k=0,\quad U_M^k=0,\quad  0\leq k\leq N.
\end{array}\end{equation}
Omitting small terms in \eqref{numerical} and \eqref{scheme-eq1}, and replacing the grid function $U_i^k$ by
its numerical approximation $u_i^k,$ we construct a second-order difference scheme for solving problem \eqref{eq}-\eqref{bz}
as follows:
\begin{align}
&\frac{1}{\Gamma(3-\alpha)}t_{1-\frac{\alpha}{3}}^{2-\alpha}\cdot\delta_t^2u_i^0
=\delta_x^2\big[(1-\frac{\alpha}{3})u_i^1+\frac{\alpha}{3}u_i^0\big]
+f_i^{1-\frac{\alpha}{3}},& 1\leq i\leq M-1,\label{scheme-eq2}
\\
&\frac{1}{\Gamma(2-\alpha)}\sum_{l=0}^kc_{l}^{(k,\alpha)}\delta_t^2u_i^{k-l}
\notag
\\
=&(\frac{1}{2}+\sigma)\delta_x^2 u_i^{k+\frac{1}{2}}+(\frac{1}{2}-\sigma)\delta_x^2 u_i^{k-\frac{1}{2}}+f_i^{k+\sigma},\,& 1\leq i\leq M-1,~1\leq k\leq N-1,\label{scheme-eq}
\\
&u_i^0=\varphi_i,& 1\leq i\leq M-1,\label{scheme-cz}
\\
&u_0^k=0,\quad u_M^k=0, &0\leq k\leq N.\label{scheme-bz}
\end{align}

Denote
\begin{equation}\label{uhat}
\widehat{u}_i^{k+1}=(\frac{3}{2}-\sigma)\big[(\frac{1}{2}+\sigma)u_i^{k+1}
+(\frac{1}{2}-\sigma)u_i^{k}\big]
-(\frac{1}{2}-\sigma)\big[(\frac{1}{2}+\sigma)u_i^{k}
+(\frac{1}{2}-\sigma)u_i^{k-1}\big].
\end{equation}

Noticing Lemma \ref{le3}, we may regard $\widehat{u}_i^{k+1}$ as an approximation of $U_i^{k+1},$ if $u_i^{k+1}, u_i^k$ and $u_i^{k-1}$ are the approximations of $U_i^{k+1}, U_i^k$ and $U_i^{k-1},$ respectively.
\subsection{Numerical stability and convergence of the difference scheme}\label{stability}
Firstly, we introduce one main lemma.
\begin{lemma}\label{positive}
For any grid functions $u^0,u^1,\cdots,u^{k+1}\in\mathring{\mathcal{U}}_h,$ we have the following inequality
\begin{equation}\label{positive3}\begin{array}{l}
\ds\hspace{0.15in} \sum_{l=0}^{k-1}c_l^{(k,\alpha)}(\delta_tu^{k-l+\frac{1}{2}}
-\delta_tu^{k-l-\frac{1}{2}},
(\frac{1}{2}+\sigma)\delta_tu^{k+\frac{1}{2}}+(\frac{1}{2}-\sigma)
\delta_tu^{k-\frac{1}{2}})
\\[0.1in]
\geq \ds \frac{1}{2}\sum_{l=0}^{k-1}c_l^{(k,\alpha)}\big(\|\delta_tu^{k-l+\frac{1}{2}}\|^2
-\|\delta_tu^{k-l-\frac{1}{2}}\|^2\big).
\end{array}\end{equation}
\end{lemma}
\begin{proof}
Borrowing the idea in \cite{alikhanov}, we can prove this lemma in three steps as follows.

\textbf{Step 1.} \begin{equation}\label{positive2}\begin{array}{l}
\ds \sum_{l=0}^{k-1}c_l^{(k,\alpha)}(\delta_tu^{k-l+\frac{1}{2}}
-\delta_tu^{k-l-\frac{1}{2}},\,\delta_tu^{k+\frac{1}{2}})\geq\\
\ds\frac{1}{2}\sum_{l=0}^{k-1}c_l^{(k,\alpha)}\big(\|\delta_tu^{k-l+\frac{1}{2}}\|^2
-\|\delta_tu^{k-l-\frac{1}{2}}\|^2\big)+\frac{1}{2c_0^{(k,\alpha)}}
\Big\|\sum_{l=0}^{k-1}c_l^{(k,\alpha)}
(\delta_tu^{k-l+\frac{1}{2}}-\delta_tu^{k-l-\frac{1}{2}})\Big\|^2.
\end{array}\end{equation}

\textbf{Step 2.} \begin{equation}\label{positive1}\begin{array}{l}
\ds\sum_{l=0}^{k-1}c_l^{(k,\alpha)}(\delta_tu^{k-l+\frac{1}{2}}
-\delta_tu^{k-l-\frac{1}{2}},\,
\delta_tu^{k-\frac{1}{2}})
\geq\frac{1}{2}\sum_{l=0}^{k-1}c_l^{(k,\alpha)}
\big(\|\delta_tu^{k-l+\frac{1}{2}}\|^2
-\|\delta_tu^{k-l-\frac{1}{2}}\|^2\big)\\
\ds-\frac{1}{2(c_0^{(k,\alpha)}-c_1^{(k,\alpha)})}
\Big\|\sum_{l=0}^{k-1}c_l^{(k,\alpha)}
(\delta_tu^{k-l+\frac{1}{2}}-\delta_tu^{k-l-\frac{1}{2}})\Big\|^2.
\end{array}\end{equation}

\textbf{Step 3.} Multiplying \eqref{positive2} by $(\frac{1}{2}+\sigma)$ and \eqref{positive1} by $(\frac{1}{2}-\sigma),$ adding the results, and noticing \eqref{positive-1}, we obtain \eqref{positive3}. The proof is ended.
\end{proof}

Now, we present the stability of the difference scheme.
\begin{theorem}\label{the1} The finite difference scheme \eqref{scheme-eq2}-\eqref{scheme-bz} is unconditionally stable.
Namely, the solution $\{u_{i}^k\,|\,0\leq i\leq M,~1\leq k\leq N\}$ can be bounded by
\begin{equation}\label{Stab-main1}\begin{array}{l}
\ds\hspace{0.15in}\frac{1}{2}(1-\frac{\alpha}{3})|u^1|_1^2
+\frac{1}{\Gamma(3-\alpha)}t_{1-\frac{\alpha}{3}}^{2-\alpha}
\|\delta_tu^\frac{1}{2}\|^2\\[0.1in]
\leq\ds\frac{2}{\Gamma(3-\alpha)}t_{1-\frac{\alpha}{3}}^{2-\alpha}\|\psi\|^2
+\big[1+\frac{\alpha^2}{3(3-\alpha)}\big]|u^0|_1^2
+\frac{\Gamma(3-\alpha)\tau^\alpha}{2(1-\alpha/3)^{2-\alpha}}
\|f^{1-\frac{\alpha}{3}}\|^2,
\end{array}\end{equation}
and,
\begin{equation}\label{stab-main2}
\frac{3T^{1-\alpha}\tau}{16\Gamma(2-\alpha)}
\sum_{l=1}^k\|\delta_tu^{l+\frac{1}{2}}\|^2+\Big|(\frac{1}{2}
+\sigma)u^{k+1}+(\frac{1}{2}-\sigma)u^{k}\Big|_1^2\leq c_3G_k,
\quad 1\leq k\leq N-1,
\end{equation}
where $c_3$ is a positive constant and
\[G_k=\|\psi\|^2+|u^0|_1^2+\tau\|\delta_t^2u^0\|^2+\tau^\alpha
\|f^{1-\frac{\alpha}{3}}\|^2
+\tau\sum_{m=1}^k\|f^{m+\sigma}\|^2,\quad 1\leq k\leq N-1.\]

Furthermore, we have
\begin{equation}\label{u-hat1}\begin{array}{l}
|\widehat{u}^{k+1}|_1^2\leq\ds 2(\frac{3}{2}-\sigma)^2\Big|(\frac{1}{2}+\sigma)u^{k+1}
+(\frac{1}{2}-\sigma)u^{k}\Big|_1^2
+2(\frac{1}{2}-\sigma)^2\Big|(\frac{1}{2}+\sigma)u^{k}+(\frac{1}{2}
-\sigma)u^{k-1}\Big|_1^2\\[0.2in]
\qquad\quad\;\leq\ds\frac{c_3}{2}(9G_k+G_{k-1})\leq 5c_3G_k,\quad 1\leq k\leq N-1,
\end{array}\end{equation}
in which $\widehat{u}^{k+1}$ can be referred to \eqref{uhat}.
\end{theorem}
\begin{proof}
\textbf{Step 1.} Taking the inner product of \eqref{scheme-eq2} with $\delta_tu^\frac{1}{2},$ and with the help of the summation by parts, we obtain
\begin{align*}
&\frac{2}{\Gamma(3-\alpha)}t_{1-\frac{\alpha}{3}}^{2-\alpha}
(\delta_tu^{\frac{1}{2}}-\psi,\delta_tu^\frac{1}{2})\\
=&-(1-\frac{\alpha}{3})(\delta_xu^1,\,\delta_xu^1-\delta_xu^0)
-\frac{\alpha}{3}(\delta_xu^0,\,\delta_xu^1-\delta_xu^0)
+\tau(f^{1-\frac{\alpha}{3}},\delta_tu^\frac{1}{2}).
\end{align*}
By the Cauchy-Schwarz inequality, it yields that
\begin{align*}
&(1-\frac{\alpha}{3})|u^1|_1^2+\frac{2}{\Gamma(3-\alpha)}
t_{1-\frac{\alpha}{3}}^{2-\alpha}
\|\delta_tu^\frac{1}{2}\|^2\\
\leq&\frac{1}{\Gamma(3-\alpha)}t_{1-\frac{\alpha}{3}}^{2-\alpha}\big(2\|\psi\|^2
+\frac{1}{2}\|\delta_tu^\frac{1}{2}\|^2\big)+(1-\frac{\alpha}{3})
\big(\frac{1}{4}|u^1|_1^2+|u^0|_1^2\big)
+\frac{\alpha}{3}\Big(\frac{3-\alpha}{4\alpha}|u^1|_1^2
+\frac{\alpha}{3-\alpha}|u^0|_1^2\Big)
\\
&+\frac{\alpha}{3}|u^0|_1^2+\Big(\frac{\Gamma(3-\alpha)\tau^2}
{2t_{1-\alpha/3}^{2-\alpha}}\|f^{1-\frac{\alpha}{3}}\|^2
+\frac{1}{2\Gamma(3-\alpha)}t_{1-\frac{\alpha}{3}}^{2-\alpha}
\|\delta_tu^\frac{1}{2}\|^2\Big),
\end{align*}
which results in
\begin{equation}\label{stab-main1}\begin{array}{l}
\ds\hspace{0.15in}\frac{1}{2}(1-\frac{\alpha}{3})|u^1|_1^2
+\frac{1}{\Gamma(3-\alpha)}t_{1-\frac{\alpha}{3}}^{2-\alpha}
\|\delta_tu^\frac{1}{2}\|^2
\\[0.1in]
\leq\ds\frac{2}{\Gamma(3-\alpha)}t_{1-\frac{\alpha}{3}}^{2-\alpha}\|\psi\|^2
+\big[1+\frac{\alpha^2}{3(3-\alpha)}\big]|u^0|_1^2
+\frac{\Gamma(3-\alpha)\tau^\alpha}{2(1-\alpha/3)^{2-\alpha}}
\|f^{1-\frac{\alpha}{3}}\|^2.
\end{array}\end{equation}
\textbf{Step 2.} Taking the inner product of \eqref{scheme-eq} with $(\frac{1}{2}+\sigma)\delta_tu^{k+\frac{1}{2}}+(\frac{1}{2}-\sigma)
\delta_tu^{k-\frac{1}{2}},$
 we have
 \begin{align*}
 &\frac{1}{\Gamma(2-\alpha)\tau}\sum_{l=0}^{k-1}c_{l}^{(k,\alpha)}
 \Big(\delta_tu^{k-l+\frac{1}{2}}
 -\delta_tu^{k-l-\frac{1}{2}},\,(\frac{1}{2}+\sigma)\delta_tu^{k+\frac{1}{2}}
 +(\frac{1}{2}-\sigma)\delta_tu^{k-\frac{1}{2}}\Big)
 \\
 &+\frac{1}{\Gamma(2-\alpha)}\Big(c_{k}^{(k,\alpha)}\delta_t^2u^0,
 \,(\frac{1}{2}+\sigma)\delta_tu^{k+\frac{1}{2}}+(\frac{1}{2}
 -\sigma)\delta_tu^{k-\frac{1}{2}}\Big)
 \\
 =&\Big((\frac{1}{2}+\sigma)\delta_x^2u^{k+\frac{1}{2}}+(\frac{1}{2}-\sigma)
 \delta_x^2u^{k-\frac{1}{2}},\,
 (\frac{1}{2}+\sigma)\delta_tu^{k+\frac{1}{2}}+(\frac{1}{2}+\sigma)
 \delta_tu^{k-\frac{1}{2}}\Big)\\
 &+(f^{k+\sigma},\,(\frac{1}{2}+\sigma)\delta_tu^{k+\frac{1}{2}}
 +(\frac{1}{2}-\sigma)\delta_tu^{k-\frac{1}{2}}),
\quad 1\leq k\leq N-1.
\end{align*}
By Lemma \ref{positive} and the summation by parts, we get
\begin{align*}
 &\frac{1}{2\Gamma(2-\alpha)\tau}\sum_{l=0}^{k-1}c_{l}^{(k,\alpha)}
 \Big(\|\delta_tu^{k-l+\frac{1}{2}}\|^2
 -\|\delta_tu^{k-l-\frac{1}{2}}\|^2\Big)
 \\
 &+\frac{1}{2\tau}\Big[\big|(\frac{1}{2}+\sigma)u^{k+1}
 +(\frac{1}{2}-\sigma)u^{k}\big|_1^2
 -\big|(\frac{1}{2}+\sigma)u^{k}+(\frac{1}{2}-\sigma)u^{k-1}\big|_1^2\Big]
 \\
 \leq&-\frac{1}{\Gamma(2-\alpha)}\Big(c_{k}^{(k,\alpha)}\delta_t^2u^0,
 \,(\frac{1}{2}+\sigma)
 \delta_tu^{k+\frac{1}{2}}+(\frac{1}{2}-\sigma)\delta_tu^{k-\frac{1}{2}}\Big)
 \\
 &+(f^{k+\sigma},\,(\frac{1}{2}+\sigma)\delta_tu^{k+\frac{1}{2}}
 +(\frac{1}{2}-\sigma)\delta_tu^{k-\frac{1}{2}}),
 \quad 1\leq k\leq N-1.
 \end{align*}
 Further, we have
 \begin{align*}
 &\ds \frac{1}{\Gamma(2-\alpha)}\sum_{l=0}^{k-1}c_{l}^{(k,\alpha)}
 \|\delta_tu^{k-l+\frac{1}{2}}\|^2
 +\big|(\frac{1}{2}+\sigma)u^{k+1}+(\frac{1}{2}-\sigma)u^{k}\big|_1^2\\
 \leq&\frac{1}{\Gamma(2-\alpha)}\Big[\sum_{l=0}^{k-2}c_{l}^{(k-1,\alpha)}
 \|\delta_tu^{k-l-\frac{1}{2}}\|^2+c_{k-1}^{(k,\alpha)}\|\delta_tu^{\frac{1}{2}}\|^2
 +\sum_{l=0}^{k-2}\big(c_{l}^{(k,\alpha)}-c_{l}^{(k-1,\alpha)}\big)
 \|\delta_tu^{k-l-\frac{1}{2}}\|^2\Big]\\
 &+\big|(\frac{1}{2}+\sigma)u^{k}+(\frac{1}{2}-\sigma)u^{k-1}\big|_1^2
 -\frac{2\tau}{\Gamma(2-\alpha)}\Big(c_{k}^{(k,\alpha)}\delta_t^2u^0,
 \,(\frac{1}{2}+\sigma)
 \delta_tu^{k+\frac{1}{2}}+(\frac{1}{2}-\sigma)\delta_tu^{k-\frac{1}{2}}\Big)\\
 &\ds+2\tau(f^{k+\sigma},\,(\frac{1}{2}+\sigma)\delta_tu^{k+\frac{1}{2}}
 +(\frac{1}{2}-\sigma)\delta_tu^{k-\frac{1}{2}}),
 \quad 1\leq k\leq N-1.
 \end{align*}
Through $\{c_l^{(k,\alpha)}\,|_{l=0}^{k-2}\}$ defined in \eqref{c-coe}, one gets $c_{l}^{(k,\alpha)}-c_{l}^{(k-1,\alpha)}=0\,\,(0\leq l\leq k-2).$

Denote
 \[F^k=\frac{1}{\Gamma(2-\alpha)}\sum_{l=0}^{k-1}c_{l}^{(k,\alpha)}
 \|\delta_tu^{k-l+\frac{1}{2}}\|^2
 +\big|(\frac{1}{2}+\sigma)u^{k+1}+(\frac{1}{2}-\sigma)u^{k}\big|_1^2,\quad 1\leq k\leq N-1.\]
 Then we have
 \begin{equation}\label{recur}\begin{array}{l}
 F^k\leq\ds F^{k-1}+\frac{1}{\Gamma(2-\alpha)}c_{k-1}^{(k,\alpha)}\|\delta_tu^{\frac{1}{2}}\|^2
 \\[0.15in]
 \ds -\frac{2\tau}{\Gamma(2-\alpha)}\Big(c_{k}^{(k,\alpha)}
 \delta_t^2u^0,\,(\frac{1}{2}+\sigma)
 \delta_tu^{k+\frac{1}{2}}+(\frac{1}{2}-\sigma)
 \delta_tu^{k-\frac{1}{2}}\Big)\\[0,15in]
 \ds +2\tau(f^{k+\sigma},\,(\frac{1}{2}+\sigma)
 \delta_tu^{k+\frac{1}{2}}
 +(\frac{1}{2}-\sigma)\delta_tu^{k-\frac{1}{2}}),
 \quad 1\leq k\leq N-1.
 \end{array}\end{equation}
Replacing $k$ by $m$ in \eqref{recur} and summing up for $m$ from $1$ to $k$ yield that
\begin{equation}\label{recur1}\begin{array}{l}
 F^k\leq\ds F^{0}+\frac{1}{\Gamma(2-\alpha)}\sum_{m=1}^kc_{m-1}^{(m,\alpha)}
 \|\delta_tu^{\frac{1}{2}}\|^2
 \\
 \ds\qquad\quad-\frac{2\tau}{\Gamma(2-\alpha)}\sum_{m=1}^k
 \Big(c_{m}^{(m,\alpha)}\delta_t^2u^0,\,(\frac{1}{2}+\sigma)
 \delta_tu^{m+\frac{1}{2}}+(\frac{1}{2}-\sigma)\delta_tu^{m-\frac{1}{2}}\Big)\\
 \ds\qquad\quad+2\tau\sum_{m=1}^k(f^{m+\sigma},\,(\frac{1}{2}+\sigma)
 \delta_tu^{m+\frac{1}{2}}+(\frac{1}{2}-\sigma)\delta_tu^{m-\frac{1}{2}}),
 \quad 1\leq k\leq N-1,
 \end{array}\end{equation}
 where
 \begin{equation*}
 F^0=\big|(\frac{1}{2}+\sigma)u^1+(\frac{1}{2}-\sigma)u^0\big|_1^2\leq 2(\frac{1}{2}+\sigma)^2|u^1|_1^2+2(\frac{1}{2}-\sigma)^2|u^0|_1^2.
 \end{equation*}
 It is known that $F^0$ is bounded (see \eqref{stab-main1}).
 By Lemma \ref{mono}, it implies that
 \[F^k>\frac{3T^{1-\alpha}\tau}{8\Gamma(2-\alpha)}\sum_{l=1}^k
 \|\delta_tu^{l+\frac{1}{2}}\|^2
 +\big|(\frac{1}{2}+\sigma)u^{k+1}+(\frac{1}{2}-\sigma)u^{k}\big|_1^2,\quad 1\leq k\leq N-1.\]
 Using the Cauchy-Schwarz inequality, we obtain
 \begin{align*}
 &\frac{3T^{1-\alpha}\tau}{8\Gamma(2-\alpha)}\sum_{l=1}^k\|
 \delta_tu^{l+\frac{1}{2}}\|^2
 +\big|(\frac{1}{2}+\sigma)u^{k+1}+(\frac{1}{2}-\sigma)u^{k}\big|_1^2\notag
 \\
 \leq&\big|(\frac{1}{2}+\sigma)u^1+(\frac{1}{2}-\sigma)u^0\big|_1^2
 +\frac{1}{\Gamma(2-\alpha)}\sum_{m=1}^kc_{m-1}^{(m,\alpha)}\|
 \delta_tu^{\frac{1}{2}}\|^2
 \notag\\
 &+\frac{2\tau}{\Gamma(2-\alpha)}\Big[\frac{64}{3T^{1-\alpha}}
 \sum_{m=1}^k(c_m^{(m,\alpha)})^2\|\delta_t^2u^0\|^2
 +\frac{3T^{1-\alpha}}{64}\sum_{m=1}^k\|\delta_tu^{m+\frac{1}{2}}\|^2\Big]\notag\\
 &+2\tau\sum_{m=1}^k\Big[\frac{64\Gamma(2-\alpha)}{3T^{1-\alpha}}\|f^{m+\sigma}\|^2
 +\frac{3T^{1-\alpha}}{64\Gamma(2-\alpha)}\|\delta_tu^{m+\frac{1}{2}}\|^2\Big]
 +\frac{3T^{1-\alpha}\tau}{32\Gamma(2-\alpha)}\|\delta_tu^\frac{1}{2}\|^2,\quad 1\leq k\leq N-1.
 \end{align*}
 It follows that
 \begin{equation*}
 \begin{array}{l}
 \ds \hspace{0.15in}\frac{3T^{1-\alpha}\tau}{16\Gamma(2-\alpha)}\sum_{l=1}^k
 \|\delta_tu^{l+\frac{1}{2}}\|^2+\big|(\frac{1}{2}+\sigma)u^{k+1}
 +(\frac{1}{2}-\sigma)u^{k}\big|_1^2
\\[0.2in]
\leq\ds\big|(\frac{1}{2}+\sigma)u^1+(\frac{1}{2}-\sigma)u^0\big|_1^2
 +\frac{1}{\Gamma(2-\alpha)}\sum_{m=1}^kc_{m-1}^{(m,\alpha)}
 \|\delta_tu^{\frac{1}{2}}\|^2
 \\[0.2in]
 \ds\hspace{0.15in}+\frac{128\tau}{3\Gamma(2-\alpha)T^{1-\alpha}}
 \sum_{m=1}^k\big(c_m^{(m,\alpha)}\big)^2\|\delta_t^2u^0\|^2
 +\frac{128\Gamma(2-\alpha)\tau}{3T^{1-\alpha}}\sum_{m=1}^k\|f^{m+\sigma}\|^2
 +\frac{3T^{1-\alpha}\tau}{32\Gamma(2-\alpha)}\|\delta_tu^\frac{1}{2}\|^2.
 \end{array}\end{equation*}
 Utilizing \eqref{sum} and \eqref{sum1}, combining the above inequality with \eqref{stab-main1} yields \eqref{stab-main2}.
 It immediately follows from \eqref{stab-main2} that \eqref{u-hat1} holds true. All this ends the proof.
\end{proof}
\begin{remark}\label{L2norm}
(1) With the help of the Cauchy-Schwarz inequality, it follows from \eqref{Stab-main1} and \eqref{stab-main2} that
\begin{equation*}
\begin{array}{l}
\ds \quad\|u^{k+1}\|^2=\|\tau\sum_{l=1}^k\delta_tu^{l+\frac{1}{2}}+u^1\|^2\leq 2\|u^1\|^2+2\tau^2\|\sum_{l=1}^k\delta_tu^{l+\frac{1}{2}}\|^2\\
\leq\ds 2\|u^1\|^2+2T\cdot\tau\sum_{l=1}^k\|\delta_tu^{l+\frac{1}{2}}\|^2
\leq\frac{L^2}{3}|u^1|_1^2+2T\tau\sum_{l=1}^k
\|\delta_tu^{l+\frac{1}{2}}\|^2\leq c_4G_k,\quad 1\leq k\leq N-1,
\end{array}\end{equation*}
where $c_4$ is a positive constant. Although both $u^{k+1}$ and $\widehat{u}^{k+1}$ approximate $U^{k+1}$ with the second-order accuracy in time, we prefer to use the latter one since the numerical stability and error estimates are in the sense of $\widehat{u}^{k+1}$ except the first level.

(2) \eqref{u-hat1} can be reduced to
\begin{equation*}
\begin{array}{l}
|\widehat{u}^{k+1}|_1\leq\ds (\frac{3}{2}-\sigma)\Big|(\frac{1}{2}+\sigma)u^{k+1}
+(\frac{1}{2}-\sigma)u^{k}\Big|_1
+(\frac{1}{2}-\sigma)\Big|(\frac{1}{2}+\sigma)u^{k}
+(\frac{1}{2}-\sigma)u^{k-1}\Big|_1\\[0.2in]
\qquad\quad\;\,\leq \ds(2-2\sigma)\sqrt{c_3G_k},\quad 1\leq k\leq N-1.
\end{array}\end{equation*}
\end{remark}

Next, we analyze the convergence of the difference scheme.
\begin{theorem}\label{convergence}
Suppose that problem \eqref{eq}-\eqref{bz} has a unique smooth solution $u\in C^{(4,4)}([0,L]\times[0,T])$ and $\{u_i^k\,|\, 0\leq i\leq M,\, 0\leq k\leq N\}$ is the solution of  difference scheme \eqref{scheme-eq2}-\eqref{scheme-bz}.
Let
\begin{equation*}
e_i^k=U_i^k-u_i^k,\quad 0\leq i\leq M,\quad 0\leq k\leq N,\quad\widehat{e}_i^k=
\begin{cases}
\ds U_i^k-\widehat{u}_i^k,\quad &0\leq i\leq M,\quad 2\leq k\leq N, \vspace{1ex}\\
\displaystyle U_i^k-u_i^k,\quad &0\leq i\leq M,\quad k=0,\,1.
\end{cases}
 \end{equation*}
Then, there exists a positive constant $c_5$ such that
\[\|\widehat{e}^{k}\|_\infty\leq c_5(\tau^2+h^2),\quad 0\leq k\leq N.\]
\end{theorem}
\begin{proof}
Subtracting \eqref{scheme-eq2}-\eqref{scheme-bz} from \eqref{numerical}, \eqref{scheme-eq1} and \eqref{czs}, respectively, we get the error system as follows:
\begin{align}
&\frac{2}{\Gamma(3-\alpha)\tau}t_{1-\frac{\alpha}{3}}^{2-\alpha}
\cdot\delta_te_i^\frac{1}{2}
=\delta_x^2\big[(1-\frac{\alpha}{3})e_i^1+\frac{\alpha}{3}e_i^0\big]
+R_i^0,\quad &1\leq i\leq M-1,\label{error-eq1}
\\
& \frac{1}{\Gamma(2-\alpha)}\Big[\sum_{l=0}^{k-1}c_{l}^{(k,\alpha)}
\delta_t^2e_i^{k-l}
+\frac{2c_k^{(k,\alpha)}}{\tau}\delta_te_i^\frac{1}{2}\Big]\notag
\\
=&(\frac{1}{2}+\sigma)\delta_x^2 e_i^{k+\frac{1}{2}}+(\frac{1}{2}-\sigma)\delta_x^2 e_i^{k-\frac{1}{2}}+R_i^k,\,
&1\leq i\leq M-1,~1\leq k\leq N-1,\label{error-eq}
\\
&e_i^0=0,\quad& 1\leq i\leq M-1,\label{error-cz}
\\
&e_0^k=0,\quad e_M^k=0,\quad& 0\leq k\leq N.\label{error-bz}
\end{align}
Noticing \eqref{error-cz}, taking the inner product of \eqref{error-eq1} with $\delta_te^\frac{1}{2},$ and applying the summation by parts and the Cauchy-Schwarz inequality, we get
\begin{align*}
&(1-\frac{\alpha}{3})|e^1|_1^2+\frac{2}{\Gamma(3-\alpha)}
t_{1-\frac{\alpha}{3}}^{2-\alpha}\|\delta_te^{\frac{1}{2}}\|^2\\
=&\tau(R^{0},\delta_te^\frac{1}{2})\leq
\frac{1}{\Gamma(3-\alpha)}t_{1-\frac{\alpha}{3}}^{2-\alpha}
\|\delta_te^{\frac{1}{2}}\|^2
+\frac{\Gamma(3-\alpha)\tau^2}{4t_{1-\frac{\alpha}{3}}^{2-\alpha}}\|R^0\|^2.
\end{align*}
The above equation leads to
\begin{align}\label{error-first}
(1-\frac{\alpha}{3})|e^1|_1^2+\frac{1}{\Gamma(3-\alpha)}
t_{1-\frac{\alpha}{3}}^{2-\alpha}\|\delta_te^{\frac{1}{2}}\|^2
\leq\frac{\Gamma(3-\alpha)\tau^2}{4t_{1-\frac{\alpha}{3}}^{2-\alpha}}\|R^0\|^2.
\end{align}
Consequently,
\[|e^1|_1=O(\tau^{4-\alpha}+h^2),\quad \tau^{1-\alpha}\|\delta_te^{\frac{1}{2}}\|=O(\tau^{4-\alpha}+h^2).\]

For $k\geq1,$ noticing \eqref{trunction}, \eqref{error1} and \eqref{error-first}, and using Theorem \ref{the1}, there exists a constant $c_6$ such that
\[\big|(\frac{1}{2}+\sigma)e^{k+1}+(\frac{1}{2}-\sigma)e^{k}\big|_1^2\leq c_3\big(\tau\|\delta_t^2e^0\|^2+\tau^\alpha\|R^0\|^2+\tau\sum_{m=1}^k\|R^m\|^2\big)
\leq c_6^2(\tau^2+h^2)^2,\quad 1\leq k\leq N-1.\]
By the embedding theorem, we get
\[\big\|(\frac{1}{2}+\sigma)e^{k+1}+(\frac{1}{2}-\sigma)e^{k}\big\|_\infty\leq \frac{\sqrt{L}}{2}\big|(\frac{1}{2}+\sigma)e^{k+1}+(\frac{1}{2}-\sigma)e^{k}\big|_1
\leq \frac{\sqrt{L}}{2}c_6(\tau^2+h^2),\quad 1\leq k\leq N-1.\]
Applying Lemma \ref{le3}, one gets
\begin{align*}
\|\widehat{e}^{k+1}\|_\infty\leq& \big(\frac{3}{2}-\sigma\big)\big\|(\frac{1}{2}+\sigma)e^{k+1}
+(\frac{1}{2}-\sigma)e^{k}\big\|_\infty
+\big(\frac{1}{2}-\sigma\big)\big\|(\frac{1}{2}+\sigma)e^{k}
+(\frac{1}{2}-\sigma)e^{k-1}\big\|_\infty
+c_7\tau^2\\
\leq &\sqrt{L}c_5(\tau^2+h^2)+c_7\tau^2,\quad 1\leq k\leq N-1,
\end{align*}
where $c_7$ is a constant. This completes the proof.
\end{proof}
\section{Finite difference scheme on the graded meshes}\label{sec-regularity}
Consider the problem \eqref{eq}-\eqref{bz}, if its solution $u\in C^{(2,2)}([0, L]\times
[0, T])$ and $f\in C^{(0,0)}([0, L]\times[0, T])$, then we have $\lim\limits_{t\to 0+}{}_C\mathrm{D}_{0,t}^\alpha u(x,t)=0$. Letting $t\to 0+$ in \eqref{eq}, the functions  $\varphi(x)$ and $f(x,0)$ must satisfy the following
necessary condition
\begin{equation}
\left\{\begin{array}{ll} -\varphi''(x)=f(x,0), & x\in
(0,L),\\
\varphi(0)=\mu(0), & \varphi(L)=\nu(0).\end{array}\right.\label{sun11}
\end{equation}
If $f(x,0)$ is known, then $\varphi(x)$ is uniquely determined by \eqref{sun11}. Conversely, if $\varphi(x)$ is known, then $\varphi(x)$ determines $f(x,0)$ uniquely. It is easy to know that $\varphi(x)$ and $f(x,0)$ are not independent of each other.

If the solution $u\in C^{(2,3)}([0, L]\times[0, T]),$  $f\in C^{(0,1)}([0, L]\times[0, T])$ and \eqref{sun11} is valid, then we can obtain $u_{tt}(x,0)=\lim\limits_{t\to 0+} u_{tt}(x,t)=0.$ Taking the derivative of \eqref{eq} with respect to $t,$ and letting $t\to 0+$ in the obtained equality, the functions $\psi(x)$ and $f_t(x,0)$ must satisfy the following necessary condition
\begin{equation}
\left\{\begin{array}{ll} -\psi''(x)=f_t(x,0), & x\in
(0,L),\\
\psi(0)=\mu_t(0), & \psi(L)=\nu_t(0).\end{array}\right.\label{sun12}
\end{equation}
If $f_t(x,0)$ is known, then $\psi(x)$ is uniquely determined by \eqref{sun12}. Conversely, if $\psi(x)$ is known, then $\psi(x)$ determines $f_t(x,0)$ uniquely. It is easy to know that $\psi(x)$ and $f_t(x,0)$ are not independent of each other.

Summarizing above results, if the smooth solution of the problem \eqref{eq}-\eqref{bz} exists, then the values of the source term $f(x,t)$ with its derivative $f_t(x,t)$ at $t=0$ have some close relations \eqref{sun11} and \eqref{sun12} with the initial values $\varphi(x)$ and $\psi(x).$ We call \eqref{sun11} and \eqref{sun12} compatibility conditions. These compatibility conditions are very similar to those of the Neumann boundary value problems of Poisson's equation.

In the above sections, we have presumed that $u(x,t)$ has suitable time regularity at the starting time $t=t_0$ (in this article, $t_0=0)$, for example, $u(x,t)\in C^{4}[t_0,T]$ for any fixed $x\in\Omega.$ This does not always hold true, that is, sometimes the solution to problem \eqref{eq}-\eqref{bz} likely behaves weak regularity at the starting time for some source term $f(x,t)$ and/or the definite conditions \eqref{cz} and \eqref{bz}, see the following regularity description of the solution to problem \eqref{eq}-\eqref{bz}.

\paragraph{Regularity of the solution} Let $\{(\lambda_i, \omega_i): i=1,2,\cdots\}$ be the eigenvalues and eigenfunctions for the Sturm-Liouville boundary value problem (see \cite{sakamoto})
\[\mathcal{L}\omega_i:=-\Delta \omega_i=\lambda_i\omega_i~\text{on}~\Omega,\quad \omega_i=0~\text{on}~\partial\Omega,\]
where the eigenfunctions are normalised by requiring $\|\omega_i\|=1$ for all $i.$ The fractional power $\mathcal{L}^\gamma$ of the operator $\mathcal{L}$ is defined for each $\gamma>0$ with domain
\[D(\mathcal{L}^\gamma):=\{g\in L_2(\Omega): \sum_{i=1}^\infty\lambda_i^{2\gamma}|(g,\omega_i)|^2<\infty\}\]
and norm $\|g\|_{\mathcal{L}^\gamma}=(\sum_{i=1}^\infty\lambda_i^{2\gamma}
|(g,\omega_i)|^2)^{1/2}.$

\begin{lemma} \rm{\cite{shen-martin-sun}}
Assume that $\varphi,\,\psi\in D(\mathcal{L}^3),$ and that source term $f\in D(\mathcal{L}^3)$ for each $t\in[0,T],$ and $f_t(\cdot,t),~f_{tt}(\cdot,t)\in D(\mathcal{L}^1)$ for each $t\in[0,T]$ with
\[\|f(\cdot,t)\|_{\mathcal{L}^3}+\|f_t(\cdot,t)\|_{\mathcal{L}^1}+
\|f_{tt}(\cdot,t)\|_{\mathcal{L}^1}\leq C\]
for some constant $C$ independent of $t.$ Then there is a unique solution $u$ of \eqref{eq}-\eqref{bz} and there exists a constant $C$ such that for all $(x,t)\in \Omega\times(0,T)$ one has
\begin{equation}\label{regularity}
\Big|\frac{\partial^ku}{\partial x^k}(x,t)\Big|\leq C,\quad \text{for}\,~ k=0,1,2,3,4; \quad
\Big|\frac{\partial^lu}{\partial t^l}(x,t)\Big|\leq C(1+t^{\alpha-l}),\quad \text{for}\,~ l=0,1,2,3.
\end{equation}
\end{lemma}
In the following, we deal with this case by using the difference scheme on the graded meshes.

Here we use a time-graded mesh $t_k=(k/N)^rT$ for $0\leq k\leq N.$ Denote $t_{k-\frac{1}{2}}=\frac{1}{2}(t_k+t_{k-1}),$ $\tau_k=t_k-t_{k-1}$ for $1\leq k\leq N;$ $t_{k+\sigma_k}=t_k+\sigma_k\tau_{k+1},$ $\rho_k=\frac{\tau_k}{\tau_{k+1}},$ $\sigma_k=(1-\frac{\alpha}{2})\rho_k,$ $\overline{\tau}_k=t_{k+1}-t_{k-1}=\tau_{k+1}+\tau_k$ for $1\leq k\leq N-1.$

For a given function $p\in C^1[0,T],$ denote $p^k=p(t_k)$ and introduce the following notations:
\[\widehat{\delta}_tp^{k-\frac{1}{2}}=\frac{p^k-p^{k-1}}{\tau_k},\quad \widehat{\delta}_t^2p^0=\frac{2}{\tau_1}
(\widehat{\delta}_tp^{\frac{1}{2}}-p'(t_0)),\quad \widehat{\delta}_t^2p^k=\frac{\widehat{\delta}_tp^{k+\frac{1}{2}}
-\widehat{\delta}_tp^{k-\frac{1}{2}}}{\overline{\tau}_k},\quad k\geq1.\]
Similar to Section \ref{appro}, on the first interval $[0,t_\frac{1}{2}],$
we have
\[\widehat{H}_{3,0}(t)=p(t_0)+p'(t_0)(t-t_0)+\frac{1}{2}
\widehat{\delta}_t^2p^0(t-t_0)^2
+\frac{1}{\overline{\tau}_1}(\widehat{\delta}_t^2p^1-\frac{1}{2}
\widehat{\delta}_t^2p^0)
(t-t_0)^2(t-t_1).\]
It is easy to know that
\begin{equation}\label{H3-graded}
\widehat{H}''_{3,0}(t)=\widehat{\delta}_t^2p^1\frac{6t-2t_1}{t_2}
+\widehat{\delta}_t^2p^0\frac{2t_{\frac{3}{2}}-3t}{t_2}.
\end{equation}
On the interval $[t_{l-\frac{1}{2}},t_{l+\frac{1}{2}}]\,(l \geq1),$ one has
\[\widehat{N}_{3,l}(t)=p^{l-1}+\widehat{\delta}_tp^{l-\frac{1}{2}}(t-t_{l-1})
+\widehat{\delta}_t^2p^{l}(t-t_{l-1})(t-t_{l})+\frac{\widehat{\delta}_t^2p^{l+1}
-\widehat{\delta}_t^2p^{l}}{t_{l+2}-t_{l-1}}(t-t_{l-1})(t-t_{l})(t-t_{l+1}).\]
Further we can get
\begin{equation}\label{N3-graded}
\widehat{N}''_{3,l}(t)=\widehat{\delta}_t^2p^{l+1}\frac{6t-2(t_{l-1}+t_{l}+t_{l+1})}
{t_{l+2}-t_{l-1}}+\widehat{\delta}_t^2p^{l}\frac{2(t_{l+2}+t_{l}+t_{l+1})-6t}
{t_{l+2}-t_{l-1}}.
\end{equation}
On the interval $[t_{k-\frac{1}{2}},t_{k+\sigma_k}],$ we obtain
\[\widehat{N}_{2,k}(t)=p^{k-1}+\widehat{\delta}_tp^{k-\frac{1}{2}}(t-t_{k-1})
+\widehat{\delta}_t^2p^{k}(t-t_{k-1})(t-t_{k}).\]
Thus we have
\begin{equation}\label{super}
\widehat{N}''_{2,k}(t)=2\widehat{\delta}_t^2p^{k}.
\end{equation}
Consequently, we can get an H3N3-2$_{\sigma_k}$ approximation on the graded meshes as follows:
\begin{equation}\label{appro-graded}\begin{array}{l}
\ds_C\mathrm{D}_{0,t}^\alpha\, p(t)|_{t=t_{k+\sigma_k}}
\\
\approx\ds\frac{1}{\Gamma(2-\alpha)}\Big[\int_{t_0}^{t_\frac{1}{2}}
\frac{\widehat{H}_{3,0}''(s)}
{(t_{k+\sigma_k}-s)^{\alpha-1}}\mathrm{d}s+\sum_{l=1}^{k-1}
\int_{t_{l-\frac{1}{2}}}^{t_{l+\frac{1}{2}}}
\frac{\widehat{N}_{3,l}''(s)}{(t_{k+\sigma_k}-s)^{\alpha-1}}\mathrm{d}s
\\
\quad+\ds\int_{t_{k-\frac{1}{2}}}^{t_{k+\sigma_k}}
\frac{\widehat{N}_{2,k}''(s)}{(t_{k+\sigma_k}-s)^{\alpha-1}}\mathrm{d}s\Big]
\\
=\ds\frac{1}{\Gamma(2-\alpha)}\sum_{l=0}^k\widetilde{c}_{l}^{(k,\alpha)}
\widehat{\delta}_t^2p^{k-l}
\equiv \mathfrak{D}p(t)|_{t=t_{k+\sigma_k}},\quad 1\leq k\leq N-1,
\end{array}\end{equation}
where
\begin{equation} \label{c-coe-non}
\widetilde{c}_{l}^{(k,\alpha)}=
\begin{cases}
\ds \widetilde{a}_{k-l}^{(k,\alpha)}+\widetilde{b}_{k-l}^{(k,\alpha)},\quad &0\leq l\leq k-1,
\vspace{1ex}\\[0.1in]
\displaystyle\int_{t_0}^{t_\frac{1}{2}}\frac{2t_\frac{3}{2}-3s}{t_2}
(t_{k+\sigma_k}-s)^{1-\alpha}\mathrm{d}s,&l=k.
\end{cases}
 \end{equation}
 Here
\begin{equation*} 
\widetilde{a}_{l}^{(k,\alpha)}=
\begin{cases}
\displaystyle\int_{t_0}^{t_\frac{1}{2}}\frac{6s-2t_{1}}{t_2}
(t_{k+\sigma_k}-s)^{1-\alpha}\mathrm{d}s,\quad &l=1,\vspace{1ex}\\
\displaystyle\int_{t_{l-\frac{3}{2}}}^{t_{l-\frac{1}{2}}}
\frac{6s-2(t_l+t_{l-1}+t_{l-2})}{t_{l+1}-t_{l-2}}
(t_{k+\sigma_k}-s)^{1-\alpha}\mathrm{d}s,\quad&2\leq l\leq k,
\end{cases}
 \end{equation*}
 and,
 \begin{equation*} 
\widetilde{b}_{l}^{(k,\alpha)}=
\begin{cases}
\displaystyle\int_{t_{l-\frac{1}{2}}}^{t_{l+\frac{1}{2}}}\frac{2(t_{l+2}
+t_{l+1}+t_l)-6s}{t_{l+2}-t_{l-1}}(t_{k+\sigma_k}-s)^{1-\alpha}\mathrm{d}s,\quad &1\leq l\leq k-1, \vspace{1ex}\\
\displaystyle\int_{t_{k-\frac{1}{2}}}^{t_{k+\sigma_k}}
2(t_{k+\sigma_k}-s)^{1-\alpha}\mathrm{d}s,\quad&l=k.
\end{cases}
 \end{equation*}
\begin{remark}
In this part, similar to the processing technique of Theorem \ref{error}, in order to improve the overall accuracy, $\sigma_k=(1-\alpha/2)\rho_k$ is selected such that the following equality holds
\begin{align*}
\int_{t_{k-\frac{1}{2}}}^{t_{k+\sigma_k}}\frac{s-t_k}
{(t_{k+\sigma_k}-s)^{\alpha-1}}\mathrm{d}s
=(\sigma_k\tau_{k+1}+\frac{\tau_k}{2})^{2-\alpha}\tau_{k+1}
\Big[\frac{\sigma_k}{2-\alpha}
-\frac{\sigma_k+\frac{\tau_k}{2\tau_{k+1}}}{3-\alpha}\Big]=0.
\end{align*}
\end{remark}

Let's consider the discretization of initial weak regular problem \eqref{eq}-\eqref{bz} at the point $t=t_{k+\sigma_k}.$ Approximate the temporal derivative by \eqref{appro-graded}. For the spatial derivative, by Taylor's expansion, we have
\begin{equation}\label{spatial-1}\begin{array}{l}
\ds \hspace{0.15in} u_{xx}(x_i,t_{k+\sigma_k})
\\[0.1in]
=\ds \frac{\frac{\tau_k}{2}+\sigma_k\tau_{k+1}}{\frac{1}{2}
(\tau_k+\tau_{k+1})}u_{xx}(x_i,t_{k+\frac{1}{2}})+ \frac{(\frac{1}{2}-\sigma_k)\tau_{k+1}}{\frac{1}{2}(\tau_k
+\tau_{k+1})}u_{xx}(x_i,t_{k-\frac{1}{2}})+O(\tau_k^2+\tau_{k+1}^2)
\\[0.2in]
=\ds \frac{\frac{\rho_k}{2}+\sigma_k}{\frac{1}{2}(\rho_{k}+1)}
\delta_x^2U_i^{k+\frac{1}{2}}+ \frac{\frac{1}{2}-\sigma_k}{\frac{1}{2}(\rho_{k}+1)}
\delta_x^2U_i^{k-\frac{1}{2}}
+O(\tau_k^2+\tau_{k+1}^2+h^2),\\[0.2in]
\ds\qquad\qquad\qquad\qquad\qquad\qquad\qquad\qquad
 1\leq i\leq M-1,\,\, 1\leq k\leq N-1.
\end{array}\end{equation}
Similar to the process of Section \ref{scheme} for solving the problem with smooth solutions, we can obtain the following difference scheme for solving the problem with initial weak regularity:
\begin{equation}\label{initial-graded}\begin{array}{l}
\ds\frac{1}{\Gamma(3-\alpha)}t_{1-\frac{\alpha}{3}}^{2-\alpha}
\widehat{\delta}_t^2u_i^0
=\delta_x^2\big[(1-\frac{\alpha}{3})u_i^1+\frac{\alpha}{3}u_i^0\big]
+f(x_i,t_{1-\frac{\alpha}{3}}),\; 1\leq i\leq M-1,
\\[0.2in]
\ds \frac{1}{\Gamma(2-\alpha)}\sum_{l=0}^k\widetilde{c}_{l}^{(k,\alpha)}
\widehat{\delta}_t^2u_i^{k-l}
\\[0.2in]
\ds=\frac{\frac{\rho_k}{2}+\sigma_k}{\frac{1}{2}(\rho_{k}+1)}
\delta_x^2u_i^{k+\frac{1}{2}}+ \frac{\frac{1}{2}-\sigma_k}
{\frac{1}{2}(\rho_{k}+1)}
\delta_x^2u_i^{k-\frac{1}{2}}+f_i^{k+\sigma_k},\;
1\leq i\leq M-1,~1\leq k\leq N-1,
\\[0.2in]
\ds u_i^0=\varphi_i,\; 1\leq i\leq M-1,
\\[0.2in]
\ds u_0^k=0,\quad u_M^k=0,\; 0\leq k\leq N.
\end{array}\end{equation}
\section{Numerical examples}\label{example}
In the section, we will display the illustrative numerical experiments to test the numerical algorithms. It is obvious from \eqref{scheme-eq2}-\eqref{scheme-bz} and \eqref{initial-graded} that the numerical solution of the current layer depends on the numerical information of all the preceding layers, which can be quite time-consuming in the actual calculations. In order to overcome this shortcoming, we use the sum-of-exponentials technique to speed up the evaluation of the H3N3-2$_\sigma$ formulae \eqref{appro1} and \eqref{appro-graded}. A key point is to effectively approximate the kernel function $(t-s)^{1-\alpha}$ by sum-of-exponentials on the interval $[\delta,T]\,(\delta$ sufficiently small).
\begin{lemma}\label{soe}\rm{\cite{jiangsd}}
For the given $\gamma\in(0,1)$ and tolerance error $\epsilon,$ cut-off time restriction $\delta$ and final time $T,$ there are one positive integer $N_{exp}^{(\gamma)},$ positive points $\{s_l^{(\gamma)}\,|\,l=1,2,\cdots,N_{exp}^{(\gamma)}\},$ and corresponding positive weights $\{w_l^{(\gamma)}\,|\,l=1,2,\cdots,N_{exp}^{(\gamma)}\}$ such that
\[|t^{-\gamma}-\sum_{l=1}^{N_{exp}^{(\gamma)}}w_l^{(\gamma)}
e^{-s_l^{(\gamma)}t}|\leq\epsilon,\quad \forall\, t\in[\delta,T],\]
where
\[N_{exp}^{(\gamma)}=O\Big((\log\frac{1}{\epsilon})(\log\log\frac{1}{\epsilon}
+\log\frac{T}{\delta})+(\log\frac{1}{\delta})(\log\log\frac{1}{\epsilon}
+\log\frac{1}{\delta})\Big).\]
\end{lemma}

Now, we will derive a fast algorithm for computing the Caputo fractional derivative $_C\mathrm{D}_{0,t}^\alpha\, p(t)|_{t=t_{k+\sigma}}.$ Using Lemma \ref{soe},
we have
\begin{equation}\label{fast-appro}\begin{array}{l}
\ds\quad _C\mathrm{D}_{0,t}^\alpha\, p(t)|_{t=t_{k+\sigma}}
\\
=\ds \frac{1}{\Gamma(2-\alpha)}\Big[\int_{t_0}^{t_\frac{1}{2}}p''(s)(t_{k+\sigma}-s)^{1-\alpha}\mathrm{d}s
\quad+\sum_{l=1}^{k-1}\int_{t_{l-\frac{1}{2}}}^{t_{l+\frac{1}{2}}}p''(s)(t_{k+\sigma}-s)^{1-\alpha}\mathrm{d}s
\\
\ds \quad+\int_{t_{k-\frac{1}{2}}}^{t_{k+\sigma}}p''(s)(t_{k+\sigma}-s)^{1-\alpha}\mathrm{d}s\Big]
\\
\approx\ds\frac{1}{\Gamma(2-\alpha)}\Big[\sum_{m=1}^{N_{exp}^{(\alpha-1)}}w_m^{(\alpha-1)}
\int_{t_0}^{t_\frac{1}{2}}H''_{3,0}(s)e^{-s_m^{(\alpha-1)}(t_{k+\sigma}-s)}\mathrm{d}s
\\
\ds\quad+\sum_{m=1}^{N_{exp}^{(\alpha-1)}}w_m^{(\alpha-1)}\sum_{l=1}^{k-1}\int_{t_{l-\frac{1}{2}}}^{t_{l+\frac{1}{2}}}
N''_{3,l}(s)e^{-s_m^{(\alpha-1)}(t_{k+\sigma}-s)}\mathrm{d}s+\int_{t_{k-\frac{1}{2}}}^{t_{k+\sigma}}\frac{N''_{2,k}(s)}
{(t_{k+\sigma}-s)^{\alpha-1}}\mathrm{d}s\Big]
\\
=\ds\frac{1}{\Gamma(2-\alpha)}\Big[\sum_{m=1}^{N_{exp}^{(\alpha-1)}}w_m^{(\alpha-1)}F_m^k
+\frac{\tau^{2-\alpha}}{2-\alpha}(\sigma+\frac{1}{2})^{2-\alpha}\delta_t^2p^k\Big]\equiv \,^F\mathcal{D}p(t)|_{t=t_{k+\sigma}},\quad 1\leq k\leq N.
\end{array}\end{equation}
Here
\begin{align*}
&F_m^k=\int_{t_0}^{t_\frac{1}{2}}H''_{3,0}(s)e^{-s_m^{(\alpha-1)}(t_{k+\sigma}-s)}\mathrm{d}s
+\sum_{l=1}^{k-1}\int_{t_{l-\frac{1}{2}}}^{t_{l+\frac{1}{2}}}N''_{3,l}(s)e^{-s_m^{(\alpha-1)}(t_{k+\sigma}-s)}\mathrm{d}s,
\quad k=1,2,\cdots.
\end{align*}
The integral $F_m^k$ for $k\geq 2$ can be evaluated by using a recursive algorithm, that is,
\begin{align*}
F_m^k=&e^{-s_m^{(\alpha-1)}\tau}F_m^{k-1}+\int_{t_{k-\frac{3}{2}}}^{t_{k-\frac{1}{2}}}N''_{3,k-1}(s)
e^{-s_m^{(\alpha-1)}(t_{k+\sigma}-s)}\mathrm{d}s\notag
\\
=&e^{-s_m^{(\alpha-1)}\tau}F_m^{k-1}+\int_{t_{k-\frac{3}{2}}}^{t_{k-\frac{1}{2}}}\Big(\delta_t^2p^{k}
\frac{s-t_{k-1}}{\tau}+\delta_t^2p^{k-1}\frac{t_{k}-s}{\tau}\Big)e^{-s_m^{(\alpha-1)}
(t_{k+\sigma}-s)}\mathrm{d}s\notag
\\
=&e^{-s_m^{(\alpha-1)}\tau}F_m^{k-1}+A_m\delta_t^2 p^{k}+B_m\delta_t^2 p^{k-1},
\end{align*}
in which
\begin{align*}
&F_m^1=\int_{t_0}^{t_\frac{1}{2}}H''_{3,0}(s)e^{-s_m^{(\alpha-1)}(t_{1+\sigma}-s)}\mathrm{d}s\notag
\\
=&\Big\{\frac{1}{s_m^{(\alpha-1)}}\big[\frac{1}{4}e^{-(\sigma+\frac{1}{2})s_m^{(\alpha-1)}\tau}+\frac{1}{2}
e^{-(\sigma+1)s_m^{(\alpha-1)}\tau}\big]\notag\\
&-\frac{3}{2(s_m^{(\alpha-1)})^2\tau}\big[e^{-(\sigma+\frac{1}{2})s_m^{(\alpha-1)}
\tau}-e^{-(\sigma+1)s_m^{(\alpha-1)}\tau}\big]\Big\}\delta_t^2p^1\notag
\\
&+\Big\{\frac{3}{s_m^{(\alpha-1)}}\big[\frac{1}{4}e^{-(\sigma+\frac{1}{2})s_m^{(\alpha-1)}\tau}
-\frac{1}{2}e^{-(\sigma+1)s_m^{(\alpha-1)}\tau}\big]\\
&
+\frac{3}
{2(s_m^{(\alpha-1)})^2\tau}\big[e^{-(\sigma+\frac{1}{2}) s_m^{(\alpha-1)}\tau}-e^{-(\sigma+1)s_m^{(\alpha-1)}\tau}\big]
\Big\}\delta_t^2p^0,
\\
A_m=&\int_{t_{k-\frac{3}{2}}}^{t_{k-\frac{1}{2}}}\frac{s-t_{k-1}}{\tau}e^{-s_m^{(\alpha-1)}(t_{k+\sigma}-s)}
\mathrm{d}s\notag
\\
=&\frac{1}{2s_m^{(\alpha-1)}}\big(e^{-(\sigma+\frac{1}{2})s_m^{(\alpha-1)}\tau}
+e^{-(\sigma+\frac{3}{2})s_m^{(\alpha-1)}\tau}\big)
-\frac{1}{(s_m^{(\alpha-1)})^2\tau}
\big(e^{-(\sigma+\frac{1}{2})s_m^{(\alpha-1)}\tau}-e^{-(\sigma+\frac{3}{2})s_m^{(\alpha-1)}\tau}\big),
\\
B_m=&\int_{t_{k-\frac{3}{2}}}^{t_{k-\frac{1}{2}}}\frac{t_{k}-s}{\tau}e^{-s_m^{(\alpha-1)}(t_{k+\sigma}-s)}
\mathrm{d}s\notag
\\
=&\frac{1}{2s_m^{(\alpha-1)}}\big(e^{-(\sigma+\frac{1}{2})s_m^{(\alpha-1)}\tau}
-3e^{-(\sigma+\frac{3}{2})s_m^{(\alpha-1)}\tau}\big)
+\frac{1}{(s_m^{(\alpha-1)})^2\tau}
\big(e^{-(\sigma+\frac{1}{2})s_m^{(\alpha-1)}\tau}-e^{-(\sigma+\frac{3}{2})s_m^{(\alpha-1)}\tau}\big).
\end{align*}

Thus, the following fast algorithm is used to compute the difference scheme \eqref{scheme-eq2}-\eqref{scheme-bz}:
\begin{align}
&\frac{1}{\Gamma(3-\alpha)}t_{1-\frac{\alpha}{3}}^{2-\alpha}\cdot\delta_t^2u_i^0
=\delta_x^2\big[(1-\frac{\alpha}{3})u_i^1+\frac{\alpha}{3}u_i^0\big]
+f(x_i,t_{1-\frac{\alpha}{3}}),\quad 1\leq i\leq M-1,\notag\\
&\frac{1}{\Gamma(2-\alpha)}\Big[\sum_{m=1}^{N_{exp}^{(\alpha-1)}}\omega_m^{(\alpha-1)}
F_{m,i}^k+\frac{\tau^{2-\alpha}}{2-\alpha}(\sigma+\frac{1}{2})^{2-\alpha}\delta_t^2u_i^k\Big]
\notag\\
&=(\frac{1}{2}+\sigma)\delta_x^2u_i^{k+\frac{1}{2}}+(\frac{1}{2}-\sigma)\delta_x^2u_i^{k-\frac{1}{2}}+f_i^{k+\sigma},\quad 1\leq i\leq M-1,\quad 1\leq k\leq N-1,\notag\\ &F_{m,i}^1=\Big\{\frac{1}{s_m^{(\alpha-1)}}\big[\frac{1}{4}e^{-(\sigma+\frac{1}{2})s_m^{(\alpha-1)}\tau}+\frac{1}{2}
e^{-(\sigma+1)s_m^{(\alpha-1)}\tau}\big]
\notag\\
&-\frac{3}{2(s_m^{(\alpha-1)})^2\tau}
\big[e^{-(\sigma+\frac{1}{2})s_m^{(\alpha-1)}
\tau}-e^{-(\sigma+1)s_m^{(\alpha-1)}\tau}\big]\Big\}\delta_t^2u_i^1
\label{fscheme-eq1}\\
&+\Big\{\frac{3}{s_m^{(\alpha-1)}}\big[\frac{1}{4}e^{-(\sigma+\frac{1}{2})s_m^{(\alpha-1)}\tau}
-\frac{1}{2}e^{-(\sigma+1)s_m^{(\alpha-1)}\tau}\big]\notag\\
&+\frac{3}{2(s_m^{(\alpha-1)})^2\tau}\big[e^{-(\sigma+\frac{1}{2}) s_m^{(\alpha-1)}\tau}-e^{-(\sigma+1)s_m^{(\alpha-1)}\tau}\big]
\Big\}\delta_t^2u_i^0,\notag\\
&\qquad\qquad\qquad\qquad\qquad\qquad\quad 1\leq m\leq N_{exp}^{(\alpha-1)},~1\leq i\leq M-1,\notag\\
& F_{m,i}^k=e^{-s_m^{(\alpha-1)}\tau}F_{m,i}^{k-1}+A_m
\delta_t^2u_i^{k}+B_{m}\delta_t^2u_i^{k-1},
\notag\\
&\qquad\qquad\qquad\qquad\qquad\qquad\qquad\quad k\geq2,\quad 1\leq m\leq N_{exp}^{(\alpha-1)},\quad 1\leq i\leq M-1,\notag\\
& u_i^0=\varphi_i,\quad 1\leq i\leq M-1,
\notag\\
& u_0^k=0,\quad u_M^k=0,\quad 0\leq k\leq N.\notag
\end{align}

Similarly, a fast algorithm for the difference scheme on graded meshes \eqref{initial-graded} is built up below:
\begin{equation}\label{fscheme-eq}
\begin{array}{l}
\ds\frac{1}{\Gamma(3-\alpha)}t_{1-\frac{\alpha}{3}}^{2-\alpha}\widehat{\delta}_t^2u_i^0
=\delta_x^2\big[(1-\frac{\alpha}{3})u_i^1+\frac{\alpha}{3}u_i^0\big]
+f(x_i,t_{1-\frac{\alpha}{3}}),\quad 1\leq i\leq M-1,
\\[0.2in]
\ds\frac{1}{\Gamma(2-\alpha)}\Big[\sum_{m=1}^{N_{exp}^{(\alpha-1)}}\omega_m^{(\alpha-1)}
\widetilde{F}_{m,i}^k+\frac{2}{2-\alpha}(t_{k+\sigma_k}-t_{k-\frac{1}{2}})^{2-\alpha}
\widehat{\delta}_t^2u_i^k\Big]
\\[0.2in]
\ds=\frac{\frac{\rho_k}{2}+\sigma_k}{\frac{1}{2}(\rho_k+1)}\delta_x^2u_i^{k+\frac{1}{2}}
+\frac{\frac{1}{2}-\sigma_k}{\frac{1}{2}(\rho_k+1)}
\delta_x^2u_i^{k-\frac{1}{2}}+f_i^{k+\sigma_k},\quad 1\leq i\leq M-1,\quad 1\leq k\leq N-1,
\\[0.2in]
\ds\widetilde{F}_{m,i}^1=\int_{t_0}^{t_\frac{1}{2}}\widehat{H}''_{3,0}(s)
e^{-s_m^{(\alpha-1)}(t_{\sigma_1+1}-s)}
\mathrm{d}s,\quad 1\leq m\leq N_{exp}^{(\alpha-1)},\quad 1\leq i\leq M-1,
\\[0.2in]
\ds\widetilde{F}_{m,i}^k=e^{-s_m^{(\alpha-1)}[(1-\sigma_{k-1})\tau_k
+\sigma_k\tau_{k+1}]}\widetilde{F}_{m,i}^{k-1}
+\widetilde{A}_m^k\widehat{\delta}_t^2u_i^{k}+\widetilde{B}_{m}^k\widehat{\delta}_t^2
u_i^{k-1},
\\[0.2in]
\ds\qquad\qquad\qquad\qquad\qquad\quad k\geq2,\quad 1\leq m\leq N_{exp}^{(\alpha-1)},
\quad 1\leq i\leq M-1,
\\[0.2in]
\ds\widetilde{A}_m^k=\int_{t_{k-\frac{3}{2}}}^{t_{k-\frac{1}{2}}}
\frac{6s-2(t_{k}+t_{k-1}+t_{k-2})}
{t_{k+1}-t_{k-2}}e^{-s_m^{(\alpha-1)}(t_{k+\sigma_k}-s)}\mathrm{d}s,
\\[0.2in]
\ds\widetilde{B}_m^k=\int_{t_{k-\frac{3}{2}}}^{t_{k-\frac{1}{2}}}
\frac{2(t_{k+1}+t_{k}+t_{k-1})-6s}
{t_{k+1}-t_{k-2}}e^{-s_m^{(\alpha-1)}(t_{k+\sigma_k}-s)}\mathrm{d}s,
\\[0.2in]
\ds u_i^0=\varphi_i,\quad 1\leq i\leq M-1,
\\[0.2in]
\ds u_0^k=0,\quad u_M^k=0,\quad 0\leq k\leq N.
\end{array}\end{equation}

Next we will display an example for solving the initial-boundary value problem \eqref{eq}-\eqref{bz} by adopting difference scheme \eqref{fscheme-eq1} and comparing it with the following L2C-based difference scheme:
\begin{equation}\label{L2C-scheme-eq2}\begin{array}{l}
\ds\frac{1}{\Gamma(3-\alpha)}t_{1-\frac{\alpha}{3}}^{2-\alpha}\cdot\delta_t^2u_i^0
=\delta_x^2\big[(1-\frac{\alpha}{3})u_i^1+\frac{\alpha}{3}u_i^0\big]
+f(x_i,t_{1-\frac{\alpha}{3}}),\quad 1\leq i\leq M-1,
\\
\ds\frac{\tau^{-\alpha}}{2\Gamma(3-\alpha)}\sum_{l=0}^{k+1}c^{(\alpha)}_{l,k}
(u_i^l-u_i^{l-1})
=\frac{1}{2}(\delta_x^2 u_i^{k+\frac{1}{2}}+\delta_x^2 u_i^{k-\frac{1}{2}})+f_i^k,
\\
\ds\qquad\qquad\qquad\qquad\qquad\qquad\qquad\qquad\qquad\qquad\qquad\,1\leq i\leq M-1,\,1\leq k\leq N-1,
\\
\ds u_i^0=\varphi_i,\quad 1\leq i\leq M-1,
\\[0.1in]
\ds u_0^k=0,\quad u_M^k=0,\quad 0\leq k\leq N,
\end{array}\end{equation}
where
\begin{align*}
c^{(\alpha)}_{l,k}=\left\{
\begin{array}{lll}
(k-l-1)^{2-\alpha}-(k-l)^{2-\alpha},&l=0,1,
\\
(k-l+2)^{2-\alpha}-(k-l+1)^{2-\alpha}-(k-l)^{2-\alpha}+(k-l-1)^{2-\alpha},&2\leq l\leq k-1,
\\
(k-l+2)^{2-\alpha}-(k-l+1)^{2-\alpha},&l=k,k+1.
\end{array}
\right.
\end{align*}
In \eqref{L2C-scheme-eq2}, $u^{-1}$ is defined as $u^{1}-2\tau \psi$ and then is used.

\begin{example}\label{ex2}
Consider the initial-boundary value problem \eqref{eq}-\eqref{bz} with $L=1,\,T=1$ and $\ds f(x,t)=\Big[\frac{24t^{5-\alpha}}{\Gamma(6-\alpha)}+\frac{\pi^2t^5}{5}\Big]\sin (\pi x).$ Then its exact solution is given as
\[u(x,t)=\frac{1}{5}t^5\sin(\pi x).\]
\end{example}

Denote the error of numerical solutions by
\[E(h,\tau)=\max_{0\leq k\leq N}\|U^k-u^k\|_\infty,\]
and the convergence order by
\[Order_\tau=\log_2\Big(\frac{E(h,2\tau)}{E(h,\tau)}\Big),\quad Order_h=\log_2\Big(\frac{E(2h,\tau)}{E(h,\tau)}\Big).\]

We adopt the fast algorithm \eqref{fscheme-eq1} by setting an absolute tolerance error $\epsilon=10^{-12}$ and cut-off time $\delta=\sigma\tau.$ We fix spacial step size $h=1/5000$ to verify the time accuracy of difference scheme \eqref{fscheme-eq1} and L2C-based difference scheme \eqref{L2C-scheme-eq2}, respectively. As we see from Table \ref{table2}, the temporal convergence order of difference scheme \eqref{fscheme-eq1} is two, while that of L2C-based difference scheme \eqref{L2C-scheme-eq2} is $3-\alpha$ for different values of $\alpha.$ Take $N$ to be $10000.$ Table \ref{table2-1} lists the spacial convergence orders for some values of $\alpha$ and $M.$ This table shows that convergence orders of difference scheme \eqref{fscheme-eq1} and L2C-based difference scheme \eqref{L2C-scheme-eq2} are two, which are in good agreement with theoretical findings.

\begin{table}[tbh!]
\setlength{\belowcaptionskip}{-5pt}
\begin{center}
\renewcommand{\arraystretch}{1.12}
\tabcolsep 0pt \caption{Truncation errors and temporal convergence
orders for Example \ref{ex2} }\label{table2}
\def\temptablewidth{0.9\textwidth}
\rule{\temptablewidth}{1pt}
{\footnotesize
\begin{tabular*}{\temptablewidth}{@{\extracolsep{\fill}}cccccccc}
\multirow{2}{*}{$\alpha$} & \multirow{2}{*}{$N$} & \multicolumn{3}{c}{difference scheme \eqref{fscheme-eq1}} & \multicolumn{3}{c}{L2C scheme \eqref{L2C-scheme-eq2}}
\\
\cline{3-5}\cline{6-8}
 ~  & ~    & $E(h,\tau)$ & $Order_\tau$ & CPU(s) & $E(h,\tau)$ & $Order_\tau$ & CPU(s)
\\
\cline{1-8}
 ~  & 160  &   1.87e-5   &      ~       & 61.11  &   3.24e-5   &     ~        & 72.28
\\
1.1 & 320  &   4.67e-6   &     1.9994   & 232.30 &   7.97e-6   &    2.0222    & 166.37
\\
~   & 640  &   1.16e-6   &     2.0056   & 350.26 &   1.96e-6   &    2.0258    & 324.69
\\
~   & 1280 &   3.14e-7   &     1.8923   & 519.77 &   4.77e-7   &    2.0347    & 648.85
\\
\cline{1-8}
~   & 160  &   2.37e-5   &       ~      & 60.32  &   6.76e-5   &      ~       & 98.66
\\
1.5 & 320  &   5.90e-6   &     2.0040   & 192.68 &   2.80e-5   &    1.2724    & 168.72
\\
~   & 640  &   1.47e-6   &     2.0049   & 393.84 &   1.09e-5   &    1.3585    & 338.83
\\
~   & 1280 &   3.65e-7   &     2.0107   & 680.51 &   4.11e-6   &    1.4074    & 666.54
\\
\cline{1-8}
~   & 160  &   2.31e-5   &       ~      & 64.39  &   1.25e-3   &     ~        & 85.25
 \\
1.9 & 320  &   5.75e-6   &     2.0097   & 202.94 &   5.90e-4   &    1.0784    & 163.24
\\
~   & 640  &   1.43e-6   &     2.0104   & 613.22 &   2.77e-4   &    1.0885    & 329.07
\\
~   & 1280 &   3.53e-7   &     2.0137   & 765.90 &   1.30e-4   &    1.0939    & 664.24
\end{tabular*}}
\rule{\temptablewidth}{1pt}
\end{center}
\end{table}
\begin{table}[tbh!]
\setlength{\belowcaptionskip}{-5pt}
\begin{center}
\renewcommand{\arraystretch}{1.12}
\tabcolsep 0pt \caption{Truncation errors and spatial convergence
orders for Example \ref{ex2}}\label{table2-1}
\def\temptablewidth{0.9\textwidth}
\rule{\temptablewidth}{1pt}
{\footnotesize
\begin{tabular*}{\temptablewidth}{@{\extracolsep{\fill}}cccccccc}
\multirow{2}{*}{$\alpha$} & \multirow{2}{*}{$M$} & \multicolumn{3}{c}{difference scheme \eqref{fscheme-eq1}} & \multicolumn{3}{c}{L2C scheme \eqref{L2C-scheme-eq2}}
\\
\cline{3-5}\cline{6-8}
~   &  ~ & $E(h,\tau)$ & $Order_h$ & CPU(s) & $E(h,\tau)$ & $Order_h$ & CPU(s)
\\
\cline{1-8}
~   &  8 &   1.56e-3   &     ~     &  0.45  &   1.56e-3   &    ~      & 71.21
\\
1.1 & 16 &   3.90e-4   &   2.0027  &  0.57  &   3.90e-4   &   2.0025  & 64.16
\\
~   & 32 &   9.74e-5   &   2.0011  &  0.59  &   9.75e-5   &   2.0007  & 84.44
\\
~   & 64 &   2.43e-5   &   2.0022  &  0.75  &   2.44e-5   &   2.0005  & 62.72
\\
\cline{1-8}
~   & 8  &   1.08e-3   &      ~    &  0.27  &   1.08e-3   &      ~    & 58.54
\\
1.5 & 16 &   2.71e-4   &   1.9994  &  0.32  &   2.71e-4   &   1.9986  & 40.84
\\
~   & 32 &   6.77e-5   &   2.0000  &  0.41  &   6.80e-5   &   1.9966  & 42.86
\\
~   & 64 &   1.69e-5   &   2.0004  &  0.55  &   1.71e-5   &   1.9869  & 46.18
\\
\cline{1-8}
~   & 8  &   6.21e-4   &     ~     &  0.28  &   6.35e-4   &      ~    & 68.83
\\
1.9 & 16 &   1.56e-4   &   1.9968  &  0.33  &   1.69e-4   &   1.9071  & 58.60
\\
~   & 32 &   3.89e-5   &   1.9993  &  0.41  &   5.25e-5   &   1.6879  & 61.13
\\
~   & 64 &   9.73e-6   &   2.0004  &  0.60  &   2.33e-5   &   1.1709  & 63.50
\end{tabular*}}
\rule{\temptablewidth}{1pt}
\end{center}
\end{table}

For completeness, the convergence orders of corresponding fast difference scheme on the graded meshes \eqref{fscheme-eq} for the initial-boundary value problem \eqref{eq}-\eqref{bz} with a weak regular solution at the initial time are tested in the next example. Specifically, we adopt the absolute tolerance error $\epsilon=10^{-12}$ and cut-off time $\delta=\tau_k/2+\sigma_k\tau_{k+1}.$
\begin{example}\label{ex3}
In this example, we consider the initial-boundary value problem \eqref{eq}-\eqref{bz} whose solution has weak regularity at the initial time. To verify the temporal accuracy when the solution has initial weak regularity, we choose the exact solution as $u(x,t)=(t^\alpha+1)\sin (\pi x),$ which has a typical regularity at $t=0$ as described in \eqref{regularity}. The source term is
\[f(x,t)=\big[\Gamma(1+\alpha)+\pi^2(t^\alpha+1)\big]\sin(\pi x).\]
\end{example}
Let the final time $T=1,$ the spatial domain $L=1.$ The error $E(M,N)=\max\limits_{1\leq k\leq N}\|u^k-U^k\|_\infty$ is recorded in each run and the experimental order is evaluated by
\[r_N=\log_2\Big(\frac{E(M,N/2)}{E(M,N)}\Big).\]
\begin{table}[tbh!]
\setlength{\belowcaptionskip}{-5pt}
\begin{center}
\renewcommand{\arraystretch}{1.12}
\tabcolsep 0pt \caption{Truncation errors and temporal convergence
orders for Example \ref{ex3}}\label{table3}
\def\temptablewidth{0.9\textwidth}
\rule{\temptablewidth}{1pt}
{\footnotesize
\begin{tabular*}{\temptablewidth}{@{\extracolsep{\fill}}cccccccc}
\multirow{2}{*}{$\alpha$} & \multirow{2}{*}{$N$} & \multicolumn{2}{c}{$r=1$} & \multicolumn{2}{c}{$r=2$} & \multicolumn{2}{c}{$r=3$}
 \\
\cline{3-4}\cline{5-6}\cline{7-8}
~   &  ~  & $E(M,N)$ & $r_N$ & $E(M,N)$ & $r_N$ & $E(M,N)$ & $r_N$
\\
\cline{1-8}
~   & 32  & 5.41e-3  &   ~   & 1.85e-3  &   ~   & 1.39e-3  & ~
\\
1.3 & 64  & 3.62e-3  & 0.58  & 1.17e-3  &  0.66 & 5.50e-4  & 1.33
\\
~   & 128 & 2.60e-3  & 0.48  & 7.63e-4  &  0.62 & 2.71e-4  & 1.02
\\
~   & 256 & 1.96e-3  & 0.40  & 5.03e-4  &  0.60 & 9.87e-5  & 1.01
\\
\cline{1-8}
~   & 32  & 3.87e-3  &   ~   & 1.03e-3  &    ~  & 1.63e-3  & ~
\\
1.5 & 64  & 2.20e-3  & 0.81  & 4.19e-4  &  1.30 & 4.25e-4  & 1.94
\\
~   & 128 & 1.32e-3  & 0.74  & 1.98e-4  &  1.08 & 1.05e-4  & 2.02
\\
~   & 256 & 8.24e-4  & 0.68  & 9.75e-5  &  1.02 & 1.54e-5  & 1.92
\\
\cline{1-8}
~   & 32  & 1.29e-3  &   ~   & 1.78e-3  &   ~   & 3.34e-3  & ~
\\
1.9 & 64  & 4.93e-4  & 1.39  & 4.48e-4  &  1.99 & 8.39e-4  & 1.99
\\
~   & 128 & 2.15e-4  & 1.20  & 1.12e-4  &  2.00 & 2.10e-4  & 2.00
\\
~   & 256 & 1.01e-4  & 1.09  & 2.81e-5  &  2.00 & 5.25e-5  & 2.00
\end{tabular*}}
\rule{\temptablewidth}{1pt}
\end{center}
\end{table}

We fix $M=5000$ to verify the time accuracy. It can be seen from Table \ref{table3} that the temporal convergence order
of the difference scheme \eqref{fscheme-eq} on the graded meshes seems to be $\min\{r(\alpha-1),2\}$ when $r$ is approximately
small and $\alpha$ is bigger, albeit no theoretical proof due to the limited length of the paper, which indicates that the
application of the graded meshes can indeed make the convergence order reach the ideal second-order accuracy.
\begin{remark}
Looking closely at Examples \ref{ex2} and \ref{ex3}, it can be seen that for the initial-boundary value problem \eqref{eq}-\eqref{bz},
when the solution has a certain smoothness, the convergence order can reach the second-order for different values of $\alpha,$ which
is consistent with the theoretical findings, while when the solution has weak regularity at the initial time, its convergence order
appears as $\min\{r(\alpha-1),2\}.$
\end{remark}

\section{Concluding remarks}
\label{sec6:conclude}
We propose a novel H3N3-2$_\sigma$ approximation formula for the Caputo derivative of order $\alpha\in(1,2)$ and derive
an accurate truncation error estimation with second-order accuracy. Then we design an H3N3-2$_\sigma$-based difference
scheme for the initial-boundary value problem of time fractional hyperbolic equation. The numerical stability analysis
and convergence in the sense of $H^1$-norm are proved by the energy method. In general, there are lots of studies on
numerical stability analysis for one-order hyperbolic equations, but those for second-order hyperbolic equations are
very limited, let alone for the fractional hyperbolic equations whose derivative orders are bigger than 1 both in time
and in space. In this sense, the present article is the first one where the numerical stability analysis is studied.
We hope that this article can open up the stability study of numerical methods for fractional hyperbolic equations.

Besides, the finite difference scheme on the graded meshes based on H3N3-2$_\sigma$ is derived. And the technique of
sum-of-exponentials is also considered in the uniform meshes and the graded meshes. The numerical simulations are presented
which support the theoretical analysis and show the advantages over the L2C-based difference scheme.

%


%
 \section*{Conflict of interest}
The authors declare that they have no conflict of interest.



\end{document}